\documentclass[a4paper, reqno]{amsart} 

\usepackage[T1]{fontenc}
\usepackage[latin1]{inputenc}
\usepackage{newlfont,verbatim, indentfirst,enumerate}
\usepackage{amsmath}
\usepackage{amssymb, amscd}
\usepackage{latexsym, amsfonts, amsbsy, mathrsfs}
\usepackage{graphicx, array, pst-all, rotating}
\usepackage{tikz} \usetikzlibrary{arrows} \usetikzlibrary{patterns}
\usepackage[unicode]{hyperref}

\renewenvironment{itemize}{\begin{list}{$\bullet$}{\leftmargin=0.5cm}\parindent=0pt}{\end{list}}

\theoremstyle{plain}
\newtheorem{theorem}{Theorem}[section]
\newtheorem{thm}[theorem]{Theorem}
\newtheorem{proposition}[theorem]{Proposition}
\newtheorem{lemma}[theorem]{Lemma}   
\newtheorem{corollary}[theorem]{Corollary}

\theoremstyle{definition}
\newtheorem{defin}[theorem]{Definition}
\newtheorem{remark}[theorem]{Remark}

\hypersetup{unicode=true, pdftoolbar=true, pdfmenubar=true, pdffitwindow=false, pdfstartview={FitH},
  pdftitle={Universal families of extensions of coherent systems}, pdfauthor={Matteo Tommasini},
  pdfkeywords={universal families, extensions, coherent systems, cohomology and base change for families
  of coherent systems},
  pdfnewwindow=true, colorlinks=false, linkcolor=red, citecolor=red, filecolor=magenta, urlcolor=cyan}

\begin{document}

\title{Universal families of extensions of coherent systems}

\author{Matteo Tommasini}

\address{\flushright
Riemann Center - Institute of Algebraic Geometry\newline
Leibniz Universit\"{a}t Hannover\newline
Welfengarten 1 - 30167 Hannover, Germany}

\email{matteo.tommasini2@gmail.com}

\date{December 1, 2012}



\thanks{This is part of the work that I did for my Ph.D.; I would like to acknowledge my advisor
Professor Peter Newstead for the opportunity to work with him, for suggesting the problem of the thesis and for his
guidance and teaching during the last year and a half. Moreover, I would also like to thank SISSA, that awarded me
with a Ph.D. fellowship for the past 3 years; I would also like to thank both the Newton Institute in Cambridge and
the University of Liverpool for their hospitality in April-June 2011 and January-March 2012 respectively.
I was also financially supported by an Erasmus fellowship for 3 months, by the Italian Indam-GNSAGA group
and by the Organizing Committee of the VBAC Conference 2012 in Barcelona. I would also like to thank the Riemann Center
and the Institute of Algebraic Geometry of Leibniz Universit\"{a}t Hannover, where I completed this paper during my
stay as a Riemann fellow.}

\begin{abstract}
We prove a result of cohomology and base change for families of coherent systems over a curve. We use that in order to
prove the existence of (non-split, non-degenerate) universal families of extensions for families of coherent systems (in
the spirit of the paper ``Universal families of extensions'' by H. Lange). Such results will be applied in subsequent
papers in order to describe the wallcrossing for some moduli spaces of coherent systems.
\end{abstract}

\maketitle

\section{Introduction}
In the last 2 decades coherent systems on algebraic curves have been widely studied in algebraic geometry, mainly
because they are a very powerful tool in order to understand Brill-Noether theory for vector bundles. In its turn,
Brill-Noether theory has an important role to play in understanding the geometric structure of the moduli space of
curves.\\

Let $C$ be any complex smooth irreducible projective curve. Then a coherent system on $C$ (see \cite{KN}) is a pair
$(E,V)$ where $E$ is a vector bundle on $C$ and $V$ is a linear subspace of the space of global sections of $E$. To any
such object one can associate a triple $(n,d,k)$ where $n$ and $d$ are the rank and the degree of $E$ respectively and
$k$ is the dimension of $V$. In order to construct a space which parametrizes coherent systems on an algebraic curve
(see \cite{KN}), one has to fix the invariants $(n,d,k)$ and also a stability parameter $\alpha$ in $\mathbb{R}$ (a
posteriori $\alpha\in\mathbb{R}_{\geq 0}$). Then one defines the $\alpha$-slope of any $(E,V)$ of type $(n,d,k)$ as

$$\mu_{\alpha}(E,V)=\mu_{\alpha}(n,d,k):=\frac{d}{n}+\alpha\,\frac{k}{n}.$$

Then there is an obvious notion of $\alpha$-(semi)stability that coincides with the usual notion of (semi)stability
for vector bundles on $C$ whenever either $V$ or $\alpha$ are zero. Having fixed these notions, one can give a natural
structure of projective (respectively quasi-projective) scheme to the sets $\widetilde{G}(\alpha;n,d,k)$ and
$G(\alpha;n,d,k)$ which parametrize $\alpha$-semistable (respectively $\alpha$-stable) coherent systems of type
$(n,d,k)$.\\

It is known (see \cite{BGMN}) that there exist only finitely many critical values where the stability condition
changes. Therefore, for every triple $(n,d,k)$ there are only finitely many distinct moduli spaces of stable coherent
systems, parametrized by open intervals of $\mathbb{R}_{\geq 0}$. A crucial issue in comparing the moduli spaces on the
left and on the right of any critical value $\alpha_c$ is that of giving a geometric description of the flip loci, i.e.
the sets of points that are added or removed by crossing $\alpha_c$. The basic description is given by
\cite[lemma 6.3]{BGMN}. Among other things, this lemma implies that any $(E,V)$ that belongs to a flip locus at a
$\alpha_c$ appears as the middle term of a non-split extension

\begin{equation}\label{1}
0\longrightarrow(E_1,V_1)\longrightarrow(E,V)\longrightarrow(E_2,V_2)\longrightarrow 0
\end{equation}
in which $(E_1,V_1)$ and $(E_2,V_2)$ are both $\alpha_c$-semistable with
  
$$\mu_{\alpha_c}(E_1,V_1)=\mu_{\alpha_c}(E,V)=\mu_{\alpha_c}(E_2,V_2).$$

The classes of extensions like (\ref{1}) are parametrized by a complex vector space

$$\mathbb{H}^1_{21}:=\operatorname{Ext}^1((E_2,V_2),(E_1,V_1)).$$

If $\textrm{Aut}(E_l,V_l)=\mathbb{C}^*$ for $l=1,2$ (this happens for example if both the $(E_l,V_l)$'s are
$\alpha_c$-stable), then the $(E,V)$'s in the middle of (\ref{1}) will be parametrized by
$\mathbb{P}(\mathbb{H}^1_{21})$. Then the basic idea in order to describe a flip locus for $(n,d,k)$ at $\alpha_c$
should simply be that of considering all invariants $(n_1,d_1,k_1),(n_2,d_2,k_2)$ such that

$$n=n_1+n_2,\quad k=k_1+k_2,\quad \mu_{\alpha_c}(n_1,d_1,k_1)=\mu_{\alpha_c}(n,d,k)=\mu_{\alpha_c}(n_2,d_2,k_2),$$
$$G_1:=\widetilde{G}(\alpha_c;n_1,d_1,k_1)\neq\varnothing\neq\widetilde{G}(\alpha_c;n_2,d_2,k_2)=:G_2$$
(the first line automatically implies that $d=d_1+d_2$). Having fixed any such data, one would like to describe a
scheme parametrizing all classes of extensions as before, letting vary the coherent systems $(E_l,V_l)\in
G_l$ for $l=1,2$. In the best possible situation the resulting scheme will consist exactly of the objects we are
interested in; otherwise one will have to remove a subscheme from it (this will be part of further papers on this
subject). So we would like to describe a fibration over $G_1\times G_2$ such that the fiber over each point
$((E_1,V_1),(E_2,V_2))$ is canonically isomorphic to $\mathbb{H}^1_{21}$ or, even better, to
$\mathbb{P}(\mathbb{H}^1_{21})$. If we denote by $g$ the genus of $C$, by using \cite[proposition 3.2]{BGMN} we get
that 

$$\textrm{dim }\mathbb{H}^1_{21}=C_{21}+\textrm{dim }\mathbb{H}^0_{21}+\textrm{dim }\mathbb{H}^2_{21},$$
where $C_{21}$ is a constant that depends only on the genus of $C$ and on the types of $(E_1,V_1)$ and $(E_2,V_2)$, 

$$\mathbb{H}^0_{21}:=\operatorname{Hom}((E_2,V_2),(E_1,V_1))\quad\textrm{and}
\quad\mathbb{H}^2_{21}:=\operatorname{Ext}^2((E_2,V_2),(E_1,V_1)).$$

Therefore in general one cannot hope to get a fibration on the whole $G_1\times G_2$, but only on each subscheme of it
where the sum of the dimension of $\mathbb{H}^0_{21}$ and $\mathbb{H}^2_{21}$ is constant (actually, it will turn out
that they need to be constant separately).\\

A slightly more complicated case arises when $\textrm{Aut}(E_1,V_1)=\mathbb{C}^*$ and $\textrm{Aut}(E_2,V_2)=
GL(t,\mathbb{C})$ for some $t\geq 2$ (or conversely); this happens when $(E_1,V_1)$ is $\alpha_c$-stable and
$(E_2,V_2)\simeq(Q,W)^{\oplus_t}$ where $(Q,W)$ is $\alpha_c$-stable. In this case we will have to take into account
an action of $GL(t,\mathbb{C})$ on $\mathbb{H}^1_{21}$, so we will need to describe bundles where the fiber over
$((E_1,V_1),(E_2,V_2))$ is canonically isomorphic to the Grassmannian $\operatorname{Grass}(t,\mathbb{H}^1_{21})$.\\

In order to give all such descriptions we will suitably adapt the results of \cite{L} about universal families of
extensions of coherent sheaves. Compared to that paper, some computations are easier because we will have to work
only with locally free sheaves all the time; other computations are more difficult since we are considering pairs of
objects of the form $(E,V)$ instead of single objects of the form $E$. In addition, we will describe also universal
families of non-degenerate extensions (see definition \ref{62}), that were not considered in that paper.\\

In particular, first of all we will prove a result of cohomology and base change for families of coherent systems in
the spirit of \cite{L}. Then we will fix any scheme $S$ of finite type over $\mathbb{C}$ and any pair of families
$(\mathcal{E}_l,\mathcal{V}_l)$  of coherent systems parametrized by $S$ for $l=1,2$ such that both

$$\textrm{dim Ext}^2((\mathcal{E}_2,\mathcal{V}_2)_s,(\mathcal{E}_1,\mathcal{V}_1)_s)\quad\textrm{and}\quad
\textrm{dim Hom}((\mathcal{E}_2,\mathcal{V}_2)_s,(\mathcal{E}_1,\mathcal{V}_1)_s)$$
are constant for all $s\in S$. In this setup we will show that there is a vector bundle $\eta:V\rightarrow S$ together
with a family of extensions

$$0\rightarrow(\eta',\eta)^*(\mathcal{E}_1,\mathcal{V}_1)\rightarrow(\mathcal{E}_V,\mathcal{V}_V)\rightarrow
(\eta',\eta)^*(\mathcal{E}_2,\mathcal{V}_2)\rightarrow 0$$
that has a universal property with respect to all extensions of the form

$$0\rightarrow(u',u)^*(\mathcal{E}_1,\mathcal{V}_1)\rightarrow(\mathcal{E}_{S'},\mathcal{V}_{S'})\rightarrow
(u',u)^*(\mathcal{E}_i,\mathcal{V}_i)\rightarrow 0$$
for all $S$-schemes $u:S'\rightarrow S$ (here $\eta'=\textrm{id}_C\times\eta$ and analogously for $u'$). We will prove analogous
results for universal families of non-split and non-degenerate extensions. Finally, we will use such results in order to describe
the schemes consisting of those coherent systems that have Jordan-H\"{o}lder filtration of length $2$ and that are added or
removed from $G(\alpha;n,d,k)$ by crossing any critical value $\alpha_c$ for a triple $(n,d,k)$.

\section{Definitions and basic facts}
Without further mention, every scheme will be of finite type over $\mathbb{C}$. $C$ will denote any complex smooth
projective irreducible curve and $g$ will denote its genus.\\

We recall (see \cite{KN}, \cite{LP}) that
a \emph{coherent system} $(E,V)$ on $C$ of type $(n,d,k)$ consists of an algebraic vector bundle $E$ over $C$, of rank
$n$ and degree $d$, and a linear subspace $V\subseteq \operatorname{H}^0(E)$ of dimension $k$. An equivalent definition
that is often used in the literature is the following. A coherent system of type $(n,d,k)$ is a triple $(E,\mathbb{V},
\phi)$ where $E$ is as before, $\mathbb{V}$ is a vector space of dimension $k$ and $\phi:\mathbb{V}\otimes
\mathcal{O}_C\rightarrow E$ is a sheaf map such that the induced morphism $\operatorname{H}^0(\phi):\mathbb{V}
\rightarrow \operatorname{H}^0(E)$ is injective. The vector space $V\subseteq \operatorname{H}^0(E)$ is then the image
$\operatorname{H}^0(\phi)(\mathbb{V})$.\\

For every scheme $S$, let us denote by $\pi_S$ the projection $C\times S\rightarrow S$; for any closed point $s$ in $S$ we
write $C_{s}$ for $C\times\{s\}$.

\begin{defin}\label{3}
({\cite[definition A.6]{BGMMN}}) A \emph{family of coherent systems of type $(n,d,k)$ on $C$ parametrized by a scheme
$S$} is any pair  $(\mathcal{E},\mathcal{V})$ where

\begin{itemize}
 \item $\mathcal{E}$ is a rank $n$ vector bundle on $C\times S$ such that $\mathcal{E}_s:=\mathcal{E}|_{C_s}$ has degree
    $d$ for all $s$ in $S$;
 \item $\mathcal{V}$ is a locally free subsheaf of ${\pi_S}_*\mathcal{E}$ of rank $k$, such that
   the fibers $\mathcal{V}_s$ map injectively to $\operatorname{H}^0(\mathcal{E}_s)$ for all $s$ in $S$.
\end{itemize}
\end{defin}

Another definition of family that appears in the literature is the following:

\begin{defin}\label{4}
(\cite[definition 2.5]{KN}) A \emph{family of coherent systems of type $(n,d,k)$ on $C$ parametrized by a scheme $S$}
is any triple $(\mathcal{E},\mathcal{V},\phi)$ where:

\begin{itemize}
 \item $\mathcal{E}$ is a rank $n$ coherent sheaf on $C\times S$, flat over $S$;
 \item $\mathcal{V}$ is a locally free sheaf on $S$ of rank $k$;
 \item $\phi:\pi_S^*\mathcal{V}\rightarrow\mathcal{E}$ is a morphism of $\mathcal{O}_{C\times S}$-modules,
\end{itemize}

such that for all $s$ in $S$ the fiber of $\phi$ over $s$ gives rise to a coherent system of type $(n,d,k)$ on 
$C_s\simeq C$.
\end{defin}

In particular, this implies that for every $s$ in $S$ the sheaf $\mathcal{E}_s$ is locally free at $(c,s)$ for all
$c\in C$, so by \cite[lemma 5.4]{N} we get that $\mathcal{E}$ is locally free. So the second definition implies the
first one. Conversely, for every family as in the first definition, one can easily associate a family according to the
second definition by considering the map $\phi$ of global sections (see also \cite[\S 3.5]{KN}). Therefore, we will use
without distinction either the first or the second definition.\\

\begin{remark}\label{5}
Both \cite{KN} and \cite{LP} allow $E$ to be any coherent sheaf in the definition of coherent systems; in \cite{LP}
$\operatorname{H}^0(\phi)$ is not required to be injective and the curve $C$ can be replaced by any projective scheme.
We will refer to such objects as \emph{weak coherent systems}. There is a definition of $\alpha$-(semi)stability for
weak coherent systems (see \cite{KN} and \cite{LP}), but we will not need to use it. We shall simply recall that on a
smooth curve a weak coherent system of type $(n,d,k)$ is $\alpha$-semistable (respectively, $\alpha$-stable) if and
only if it is (the evaluation map of) an $\alpha$-semistable (respectively, $\alpha$-stable) coherent system of type
$(n,d,k)$ (see \cite[lemma 2.5]{KN}), so this makes no difference. Similarly, there is a more general notion of family
of coherent systems that is used in \cite{LP} and in \cite{He}. In the case when their base $X$ is a projective curve
$C$ and we have a condition of flatness (see \cite[\S 1.3]{He} and \cite{LP}), we get that the notion of ``flat family
of coherent systems on $X\times S/S$'' in \cite{He} coincides with the notion of ``family of coherent systems''
parametrized by $S$ given in the previous definitions.
\end{remark}

A morphism of families of coherent systems $(\mathcal{E}_2,\mathcal{V}_2,\phi_2)\rightarrow(\mathcal{E}_1,
\mathcal{V}_1,\phi_1)$ parametrized by a scheme $S$ is any pair of morphisms $(\gamma,\delta)$ where $\gamma$ is a
morphism of vector bundles $\mathcal{E}_2\rightarrow\mathcal{E}_1$ over $C\times S$ and $\delta$ is a morphisms of
vector bundles $\mathcal{V}_2\rightarrow\mathcal{V}_1$ over $S$, such that we have a commutative diagram as follows:

\[\begin{tikzpicture}[xscale=1.5,yscale=-1.2]
    \node (A0_0) at (0, 0) {$\pi_S^*\mathcal{V}_2$};
    \node (A0_2) at (2, 0) {$\mathcal{E}_2$};
    \node (A1_1) at (1, 1) {$\curvearrowright$};
    \node (A2_0) at (0, 2) {$\pi_S^*\mathcal{V}_1$};
    \node (A2_2) at (2, 2) {$\mathcal{E}_1.$};
    \path (A0_0) edge [->]node [auto,swap] {$\scriptstyle{\pi_S^*\delta}$} (A2_0);
    \path (A2_0) edge [->]node [auto,swap] {$\scriptstyle{\phi_1}$} (A2_2);
    \path (A0_2) edge [->]node [auto] {$\scriptstyle{\gamma}$} (A2_2);
    \path (A0_0) edge [->]node [auto] {$\scriptstyle{\phi_2}$} (A0_2);
\end{tikzpicture}\]

We denote by

$$\operatorname{Hom}_S((\mathcal{E}_2,\mathcal{V}_2,\phi_2),(\mathcal{E}_1,\mathcal{V}_1,\phi_1))$$
the set of all such morphisms. If we use the first definition of coherent systems, we will simply write
$\operatorname{Hom}_S((\mathcal{E}_2,\mathcal{V}_2),(\mathcal{E}_1,\mathcal{V}_1))$. A morphism between 2 families of
coherent systems over $S=\textrm{Spec}(\mathbb{C})$ is any morphism $(E_2,V_2)\rightarrow(E_1,V_1)$ between coherent
systems. As such, it is completely determined as a morphism $\gamma:E_1\rightarrow E_2$ such that $\operatorname{H}^0
(\gamma)(V_2)\subset V_1$. In this case we will write $\operatorname{Hom}((E_2,V_2),(E_1,V_1))$ for the vector space
of all such morphisms.\\

For every family of coherent systems $(\mathcal{E},\mathcal{V})$ of type $(n,d,k)$ parametrized by $S$ and for every
morphism of schemes $u:S'\rightarrow S$, the pullback via $u$ is defined as

\begin{equation}\label{7}
(u',u)^*(\mathcal{E},\mathcal{V}):=(u'^*\mathcal{E},u^*\mathcal{V}),
\end{equation}
where $u'$ is defined by the cartesian diagram

\begin{equation}\label{8}
\begin{tikzpicture}[xscale=1.5,yscale=-1.2]
    \node (A0_0) at (0, 0) {$C\times S'$};
    \node (A0_2) at (2, 0) {$C\times S$};
    \node (A1_1) at (1, 1) {$\square$};
    \node (A2_0) at (0, 2) {$S'$};
    \node (A2_2) at (2, 2) {$S.$};
    \path (A2_0) edge [->]node [auto,swap] {$\scriptstyle{u}$} (A2_2);
    \path (A0_0) edge [->]node [auto,swap] {$\scriptstyle{\pi_{S'}}$} (A2_0);
    \path (A0_2) edge [->]node [auto] {$\scriptstyle{\pi_S}$} (A2_2);
    \path (A0_0) edge [->]node [auto] {$\scriptstyle{u'}$} (A0_2);
\end{tikzpicture}
\end{equation}

It easy to see that (\ref{7}) is a family of coherent systems of type $(n,d,k)$ on $C$, parametrized by $S'$. If we
use the definition of family as triple $(\mathcal{E},\mathcal{V},\phi)$, then the pullback of such a family by $u$ is
the triple $(u',u)^*(\mathcal{E},\mathcal{V},\phi):=(u'^*\mathcal{E},u^*\mathcal{V},\widetilde{\phi})$, where
$\widetilde{\phi}$ is defined as the composition

$$\widetilde{\phi}:\,\pi_{S'}^*u^*\mathcal{V}\stackrel{\sim}{\longrightarrow}u'^*\pi_S^*\mathcal{V}
\stackrel{u'^*\phi}{\longrightarrow}u'^*\mathcal{E},$$
where the first map is the canonical isomorphism induced by diagram (\ref{8}).\\

Given any family $(\mathcal{E},\mathcal{V},\phi)$ of type $(n,d,k)$ parametrized by a scheme $S$ and any locally free
$\mathcal{O}_S$-module $\mathcal{M}$, we define

$$(\mathcal{E},\mathcal{V},\phi)\otimes_S \mathcal{M}:=(\mathcal{E}\otimes_{C\times S}\pi_S^*\mathcal{M},\mathcal{V}
\otimes_S\mathcal{M},\phi\,\otimes_{C\times S}\,\textrm{id}_{\pi_S^*\mathcal{M}}).$$

This is a again a family of coherent systems parametrized by $S$.\\

\begin{remark}\label{9}
If $\mathcal{M}$ is only a coherent or quasi-coherent $\mathcal{O}_S$-module, then the tensor product
$(\mathcal{E},\mathcal{V},\phi)\otimes_S\mathcal{M}$ in general is only a family of weak coherent systems. To be more
precise, it is an algebraic system on $C\times S/S$ in the sense of \cite{He}. One should also need to consider such
objects in order to define the functors $\operatorname{Ext}^i$'s (see below), but we will not need to deal explicitly
with such objects in the present work.
\end{remark}

For every parameter $\alpha\in\mathbb{R}$ and for every coherent system $(E,V)$ of type $(n,d,k)$, the $\alpha$-slope
of $(E,V)$ is defined as

$$\mu_{\alpha}(E,V)=\mu_{\alpha_c}(n,d,k):=\frac{d}{n}+\alpha\,\frac{k}{n}.$$

Given a coherent system $(E,V)$, a coherent subsystem is a pair $(E',V')$ such that $E'$ is a subbundle of $E$ and
$V'\subseteq V\cap \operatorname{H}^0(E')$. We say that $(E,V)$ is $\alpha$-stable if

$$\mu_{\alpha}(E',V') < \mu_{\alpha}(E,V)$$
for all proper subsystems $(E',V')$ (i.e. those such that $(0,0)\subsetneq(E',V')\subsetneq(E,V)$). The notion of
$\alpha$-semistability is obtained by replacing the strict inequality before by a weak inequality. Whenever $k\geq 1$
there are no $\alpha$-semistable coherent systems of type $(n,d,k)$ for $\alpha<0$, so a posteriori we restrict to
$\alpha\in\mathbb{R}_{\geq 0}$. It is known (\cite[corollary 2.5.1]{KN}) that the $\alpha$-semistable coherent systems
of any fixed $\alpha$-slope form a noetherian and artinian abelian category in which the simple objects are the
$\alpha$-stable coherent systems. Any $\alpha$-semistable coherent system has an $\alpha$-Jordan-H\"{o}lder filtration;
in general this filtration is not unique; however the graded objects associated to different filtrations of the same
object are always isomorphic. We recall also:

\begin{thm}(\cite[theorem 1]{KN})
For every parameter $\alpha\in\mathbb{R}_{\geq 0}$ and for every type $(n,d,k)$ there exist schemes $G(\alpha;n,d,k)$
and $\widetilde{G}(\alpha;n,d,k)$ which are coarse moduli spaces for families of $\alpha$-stable (respectively
$\alpha$-semistable) coherent systems of type $(n,d,k)$. The closed points of $G(\alpha;n,d,k)$ are in bijection with
isomorphism classes of $\alpha$-stable coherent systems. The closed points of $\widetilde{G}(\alpha;n,d,k)$ are in
bijection with $S$-equivalence classes of $\alpha$-semistable coherent systems. $\widetilde{G}(\alpha;n,d,k)$ is a
projective variety and it contains $G(\alpha;n,d,k)$ as an open subscheme.
\end{thm}

\begin{remark}\label{10} 
For each $(n,d,k)$ and $\alpha\geq 0$, the proof of this theorem follows from a GIT construction: there exist a
projective scheme $R$ and an action of $PGL(N)$ on $R$ (both $R$ and $N$ depend on $(n,d,k)$), together with a
linearization of that action depending on $\alpha$. Then if we denote by $\hat{G}(\alpha;n,d,k)$ the subscheme of
GIT $\alpha$-stable points of $R$, we get that the moduli space $G(\alpha;n,d,k)$ is obtained as the quotient
$\hat{G}(\alpha;n,d,k)/PGL(N)$. In particular, there exists a family $(\mathcal{Q},\mathcal{W})$ parametrized by
$\hat{G}(\alpha;n,d,k)$, that has the local universal property (see \cite[\S 3.5]{KN}). Analogous results holds for
the moduli scheme of semistable objects $\widetilde{G}(\alpha;n,d,k)$.
\end{remark}

Let us fix any triple $(n,d,k)$: for numerical reasons there are finitely many critical values $\{\alpha_0=0 < \alpha_1
< \cdots < \alpha_L\}\subset\mathbb{R}_{\geq 0}$ such that $G(\alpha;n,d,k)\simeq G(\alpha';n,d,k)$ for all $\alpha,
\alpha'\in\,]\alpha_i,\alpha_{i+1}[$ for all $i\in \{0,\cdots,L-1\}$ (see \cite{BGMN} for details).\\

The following definition is taken from \cite[\S 1.2]{He}. In that paper the definition is given for families of
algebraic systems; we state only the definition for the case of families of coherent systems.

\begin{defin}
Let $S$ be any scheme, let $(\mathcal{E}_l,\mathcal{V}_l)$ for $l=1,2$ be two families of coherent systems parametrized
by $S$ and let us denote by $\pi_S$ the projection $C\times S\rightarrow S$. Then we define a sheaf of
$\mathcal{O}_S$-modules

$$\mathcal{F}=\mathcal{H}om_{\pi_S}((\mathcal{E}_2,\mathcal{V}_2),(\mathcal{E}_1,\mathcal{V}_1))$$
as follows: for every open set $U\subset S$ we set 

$$\mathcal{F}(U):=\operatorname{Hom}_U((\mathcal{E}_2,\mathcal{V}_2)|_U,(\mathcal{E}_1,\mathcal{V}_1)|_U)=
\operatorname{Hom}_U\left((\mathcal{E}_2|_{\pi_S^{-1}U},\mathcal{V}_2|_U),(\mathcal{E}_1|_{\pi_S^{-1}U},
\mathcal{V}_1|_U)\right).$$

This is actually a sheaf and the functor $\mathcal{H}om_{\pi_S}((\mathcal{E}_2,\mathcal{V}_2),-)$ is left exact.
We denote by

$$\mathcal{E}xt^i_{\pi_S}((\mathcal{E}_2,\mathcal{V}_2),-)\quad\forall\,i\geq 1$$
its right derived functors (see \cite{He} for the proof that such derived functors exist). Moreover, we will denote by
$\operatorname{Ext}^i_S((\mathcal{E}_2,\mathcal{V}_2),-)$ the right derived functors of the functor
$\operatorname{Hom}_S((\mathcal{E}_2,\mathcal{V}_2),-)$. If $S=\textrm{Spec}(\mathbb{C})$, then the 2 functors
$\mathcal{E}xt^i_{\pi_S}((\mathcal{E}_2,\mathcal{V}_2),-)$ and $\operatorname{Ext}^i_S((\mathcal{E}_2,
\mathcal{V}_2),-)$ coincide for every $i$ and we denote them by $\operatorname{Ext}^i((E_2,V_2),-)$. In general their
relationship is accounted for by a spectral sequence, see \cite[prop. A.9]{BGMMN}. 
\end{defin}

By using remark \ref{5} and \cite[corollaire 1.20]{He} for the morphism $\pi_S$ we have the following useful result
of semicontinuity.

\begin{proposition}\label{11}
Let $(\mathcal{E}_l,\mathcal{V}_l)$ be two families of coherent systems (not necessarily of the same type),
parametrized by a scheme $S$ for $l=1,2$. Then for all $i\geq 0$ the function

$$t^i(s):=\operatorname{dim }\operatorname{Ext}^i((\mathcal{E}_2,\mathcal{V}_2)_s,(\mathcal{E}_1,\mathcal{V}_1)_s)$$
is upper semicontinuous on $S$. If $S$ is integral and for a certain $i$ the function $t^i(s)$ is constant on $S$, then
the sheaf $\mathcal{E}xt^i_{\pi_S}((\mathcal{E}_2,\mathcal{V}_2),(\mathcal{E}_1,\mathcal{V}_1))$ is locally free on $S$.
\end{proposition}

Let us fix any triple $(n,d,k)$, any critical value $\alpha_c$ for that triple and any $(E,V)$ that is stable only on
one side of $\alpha_c$. Then it is easy to see that $(E,V)$ is strictly $\alpha_c$-semistable and so it has an
$\alpha_c$-Jordan-H\"{o}lder filtration of length $\geq 2$. The final aim of this paper is to describe the schemes of
those $(E,V)$'s that are stable only on one side of $\alpha_c$ and that have $\alpha_c$-Jordan-H\"{o}lder filtration
of length $2$. We denote by $G^{+,2}(\alpha_c;n,d,k)$, respectively by $G^{-,2}(\alpha_c;n,d,k)$, the set of those
$(E,V)$'s that have $\alpha_c$-Jordan-H\"{o}lder filtration of length $2$ and such that $(E,V)\in G(\alpha_c^+;n,d,k)$
and $(E,V)\not\in G(\alpha_c^-;n,d,k)$, respectively $(E,V)\in G(\alpha_c^-;n,d,k)$ and $(E,V)\not\in
G(\alpha_c^+;n,d,k)$.\\

As a consequence of \cite[lemma 6.3]{BGMN} we get:

\begin{lemma}\label{91}
Let us fix any triple $(n,d,k)$ and any critical value $\alpha_c$ for it. Let us suppose that $(E,V)\in
G^{+,2}(\alpha_c;n,d,k)$, respectively that $(E,V)\in G^{-,2}(\alpha_c;n,d,k)$. Then $(E,V)$ is associated to a unique
class of a non-split extension

\begin{equation}\label{12}
0\longrightarrow(E_1,V_1)\stackrel{\gamma}{\longrightarrow}(E,V)\stackrel{\delta}{\longrightarrow}(E_2,V_2)
\longrightarrow 0,
\end{equation}
modulo the action of $\mathbb{C}^*$ (given by $\gamma\mapsto\lambda\cdot\gamma$ for $\lambda\in\mathbb{C}^*$)) in which:

\begin{itemize}
 \item $(E_1,V_1)$ and $(E_2,V_2)$ are uniquely determined by $(E,V)$; we denote by $(n_l,d_l,k_l)$ their types for
   $l=1,2$;
 \item both $(E_1,V_1)$ and $(E_2,V_2)$ are $\alpha_c$-stable with
 
  \begin{equation}\label{13}
  d_1=n_1\left(\frac{d}{n}+\alpha_c\left(\frac{k}{n}-\frac{k_1}{n_1}\right)\right),
  \end{equation}
  $$\frac{k_1}{n_1}<\frac{k}{n}\quad\textrm{respectively}\quad\frac{k_1}{n_1}>\frac{k}{n}.$$
\end{itemize}

Conversely, let us fix any pair $(n_1,k_1)$ such that 

$$0<n_1<n,\quad 0\leq k_1\leq k,$$
$$\frac{k_1}{n_1}<\frac{k}{n}\quad\textrm{respectively}\quad\frac{k_1}{n_1}>\frac{k}{n}$$
and let us set $d_1$ as in (\ref{13}); moreover let us set $n_2:=n-n_1, d_2:=d-d_1, k_2:=k-k_1$. Then for every pair
of $\alpha_c$-stable coherent systems $(E_l,V_l)$ of type $(n_l,d_l,k_l)$ for $l=1,2$ we have that any $(E,V)$ that
appears in a non-split sequence (\ref{12}) belongs to $G^{+,2}(\alpha_c;n,d,k)$, respectively to
$G^{-,2}(\alpha_c;n,d,k)$.
\end{lemma}

\begin{defin}\label{14}
Let us fix any triple $(n,d,k)$ and any critical value $\alpha_c$ for it; moreover let us also fix any pair $(n_1,k_1)$
with $0<n_1<n$, $0\leq k_1\leq k$. Let us set $d_1$ as in (\ref{13}), $n_2:=n-n_1$, $d_2:=d-d_1$ and $k_2:=k-k_1$.
Then we denote by $G(\alpha_c;n,d,k;n_1,k_1)$ the set of all triples $((E_1,V_1),(E_2,V_2),[\xi])$ such that
$(E_l,V_l)\in G(\alpha_c;n_l,d_l,k_l)$ for $l=1,2$ and $[\xi]\in\mathbb{P}(\operatorname{Ext}^1((E_2,V_2),(E_1,V_1)))$. 
\end{defin}

Note that the invariant $d_1$ defined as in (\ref{13}) (and so also $d_2$) can be a non-integer; if this happens, then
it means that there are no $(E_l,V_l)$'s of type $(n_l,d_l,k_l)$ and so the set $G(\alpha_c;n,d,k;n_1,k_1)$ is empty.
For every $(E,V)$ in $G^{+,2}(\alpha_c;n,d,k)$ or $G^{-,2}(\alpha_c;n,d,k)$ the coherent subsystem $(E_1,V_1)$ of the
previous lemma is completely determined by $(E,V)$ and so is in particular $(n_1,k_1)$. Therefore we get:

\begin{corollary}\label{15}
For any triple $(n,d,k)$ and any critical value $\alpha_c$ for it, the set $G^{+,2}(\alpha_c;n,d,k)$ has a
stratification

$$G^{+,2}(\alpha_c;n,d,k)=\coprod_{(n_1,k_1)} G(\alpha_c;n,d,k;n_1,k_1)$$
where the disjoint union is taken over all the pairs $(n_1,k_1)$ such that 

$$0<n_1<n,\quad 0\leq k_1\leq k,\quad\frac{k_1}{n_1}<\frac{k}{n},\quad G(\alpha_c;n_1,d_1,k_1)\neq\varnothing\neq
G(\alpha_c;n_2,d_2,k_2).$$

Analogously,

$$G^{-,2}(\alpha_c;n,d,k)=\coprod_{(n_1,k_1)} G(\alpha_c;n,d,k;n_1,k_1)$$
where the disjoint union is taken over all the pairs $(n_1,k_1)$ such that 

$$0<n_1<n,\quad 0\leq k_1\leq k,\quad\frac{k_1}{n_1}>\frac{k}{n},\quad G(\alpha_c;n_1,d_1,k_1)\neq\varnothing\neq
G(\alpha_c;n_2,d_2,k_2).$$
\end{corollary}

The motivation for this paper is that of giving a scheme theoretic description of each set of the form $G(\alpha_c;
n,d,k;n_1,k_1)$ and to prove that if $\frac{k_1}{n_1}<\frac{k}{n}$, respectively if $\frac{k_1}{n_1}>\frac{k}{n}$, then
$G(\alpha_c;n,d,k;n_1,k_1)$ is actually a subscheme of $G(\alpha_c^+;n,d,k)$, respectively of $G(\alpha_c^-;n,d,k)$.
In order to do that we will need the following result.

\begin{proposition}\label{16}
\cite[proposition 3.2]{BGMN}
Let $(E_l,V_l)$ be two coherent systems on $C$ of type $(n_l,d_l,k_l)$ for $l=1,2$. Let

$$\mathbb{H}^0_{21}:=\operatorname{Hom}((E_2,V_2),(E_1,V_1))\quad\textrm{and}\quad\mathbb{H}^2_{21}:=
\operatorname{Ext}^2((E_2,V_2),(E_1,V_1));$$
then:

$$\operatorname{dim }\operatorname{Ext}^1((E_2,V_2),(E_1,V_1))=C_{21}+\operatorname{dim }\mathbb{H}^0_{21}+
\operatorname{dim }\mathbb{H}^2_{21},$$
where

$$C_{21}:=n_1n_2(g-1)-d_1n_2+d_2n_1+k_2d_1-k_2n_1(g-1)-k_1k_2.$$
\end{proposition}

\section{Cohomology and base change for families of coherent systems}

We want to prove a series of statements analogous to those in \cite{L} for families of extensions of coherent systems
instead of families of extensions of coherent sheaves. The statements of \cite{L} are true for every projective
morphism $f:X\rightarrow Y$. In the present work we have to restrict to the case when $f$ is the projection $\pi_S:C
\times S\rightarrow S$ for any scheme $S$ of finite type over $\mathbb{C}$ because we have to use \cite[proposition
1.13]{He}, that is not proved in full generality. It seems possible to prove results analogous to those of \cite{L} in
full generality, but this will require more work. Note that as in \cite{L} we need a flatness hypothesis on the
families we will use. Such an hypothesis is implicit in the definition of families of coherent systems, see remark
\ref{5}.\\

Almost all the results of this section are true even if the curve $C$ is replaced by any projective scheme, once we
enlarge the notion of coherent systems and families of such objects, see again remark \ref{5}.\\

In this section we will have to consider every family of coherent systems as a triple as in definition \ref{4}. Let
us first state the following preliminary result.

\begin{proposition}\label{17}
Let $(\mathcal{E}_l,\mathcal{V}_l,\phi_l)$  be two families of coherent systems over $C$, parametrized by a scheme $S$
for $l=1,2$. Let us fix also any $S$-scheme $u:S'\rightarrow S$. Then there exists a resolution $\Delta_{\bullet}
\rightarrow(\mathcal{E}_2,\mathcal{V}_2,\phi_2)$ such that:

\begin{enumerate}[(i)]
 \item $\Delta_0=(\mathcal{P}_0,0,0)\oplus(\pi_S^*\mathcal{V}_2,\mathcal{V}_2,\textrm{id}_{\pi_S^*\mathcal{V}_2})$
 \item $\Delta_j=(\mathcal{P}_j,0,0)$ for all $j\geq 1$;
 \item $\mathcal{P}_j$ is a locally free $\mathcal{O}_{C\times S}$-module for all $j\geq 0$;
 \item for all locally free $\mathcal{O}_{S'}$-modules $\mathcal{M}$, for all $j\geq 0$ and for all $i\geq 1$
    we have $\mathcal{E}xt^i_{\pi_{S'}}((u',u)^*\Delta_j,$ $(u',u)^*(\mathcal{E}_1,\mathcal{V}_1,\phi_1)\otimes_{S'}
   \mathcal{M})=0$;
 \item for all $j\geq 0$ the sheaf $\mathcal{L}^j:=\mathcal{H}om_{\pi_S}(\Delta_j,(\mathcal{E}_1,\mathcal{V}_1,
   \phi_1))$ is locally free on $S$.
\end{enumerate}

\end{proposition}

\begin{proof}
The proof consists simply in combining the proof of \cite[proposition 1.13]{He} with the proof of \cite[lemma 1.1]{L}.
Actually, point (iv) holds for every quasi-coherent $\mathcal{O}_{S'}$-module $\mathcal{M}$, once we suitably enlarge
the category of coherent systems in order to take into account also algebraic systems (see remark \ref{9}).
\end{proof}

\begin{defin}
In the notation of \cite{He}, a \emph{very negative resolution of }$(\mathcal{E}_2,\mathcal{V}_2,\phi_2)$
\emph{with respect to} $(\mathcal{E}_1,\mathcal{V}_1,\phi_1)$ is any resolution $\Delta_{\bullet}$ of
$(\mathcal{E}_2,\mathcal{V}_2,\phi_2)$ with properties (i), (ii), (iii) and:

\begin{enumerate}[(i)]
\setcounter{enumi}{3}
 \item\hspace{-2mm}' $\mathcal{E}xt^i_{\pi_S}(\Delta_j,(\mathcal{E}_1,\mathcal{V}_1,\phi_1))=0$ for all $j\geq 0$ and
  for all $i\geq 1$.
\end{enumerate}
\end{defin}

\begin{remark}\label{18}
The previous proposition proves that if we fix any morphism $u:S'\rightarrow S$ and any pair of families
$(\mathcal{E}_l,\mathcal{V}_l,\phi_l)$ for $l=1,2$, then $(u',u)^*\Delta_{\bullet}$ is a very negative resolution of
$(u',u)^*(\mathcal{E}_2,\mathcal{V}_2,\phi_2)$ with respect to $(u',u)^*(\mathcal{E}_1,\mathcal{V}_1,\phi_1)
\otimes_{S'}\mathcal{M}$ for all locally free $\mathcal{O}_{S'}$-modules $\mathcal{M}$. In particular, if we choose
$u=\textrm{id}_S$ and $\mathcal{M}=\mathcal{O}_S$, we get a very negative resolution of $(\mathcal{E}_2,\mathcal{V}_2,
\phi_2)$ with respect to $(\mathcal{E}_1,\mathcal{V}_1,\phi_1)$.
\end{remark}

We recall the following result, obtained from \cite[remarque 1.15]{He} together with remark  \ref{5}.

\begin{lemma}\label{19}
For every scheme $S$, for every pair of families $(\mathcal{E}_l,\mathcal{V}_l,\phi_l)$ of coherent systems
parametrized by $S$ for $l=1,2$, for every very negative resolution $\Delta_{\bullet}$ of $(\mathcal{E}_2,
\mathcal{V}_2,\phi_2)$ with respect to $(\mathcal{E}_1,\mathcal{V}_1,\phi_1)$ and for every $i\geq 0$ we have a
canonical isomorphism of sheaves over $S$:

$$\mathcal{E}xt^i_{\pi_S}\Big((\mathcal{E}_2,\mathcal{V}_2,\phi_2),(\mathcal{E}_1,\mathcal{V}_1,\phi_1)\Big)=
\mathcal{H}^i\Big(\mathcal{H}om_{\pi_S}\Big(\Delta_{\bullet},(\mathcal{E}_1,\mathcal{V}_1,\phi_1)\Big)\Big).$$
\end{lemma}

Now let us fix any $u:S'\rightarrow S$ and any $2$ families parametrized by $S$ as before; let $\Delta_{\bullet}$
and $\mathcal{L}^{\bullet}$ be as in proposition \ref{17}. Then we have an analogue of \cite[corollary 1.2 (ii)]{L}
as follows.

\begin{lemma}\label{20}
For all $S$-schemes $u: S'\rightarrow S$, for all locally free $\mathcal{O}_{S'}$-modules $\mathcal{M}$ and for all
$j\geq 0$ there is a canonical isomorphism of $\mathcal{O}_{S'}$-modules:

\begin{equation}\label{21}
\mathcal{H}om_{\pi_{S'}}\Big((u',u)^*\Delta_j,(u',u)^*(\mathcal{E}_1,\mathcal{V}_1,\phi_1)\otimes_{S'}\mathcal{M}\Big)
\simeq u^*\mathcal{L}^j\otimes_{S'}\mathcal{M}.
\end{equation}
\end{lemma}

\begin{proof}
We have to consider two different cases depending on $j$.\\

\textbf{Case (i)} Let us suppose that $j\geq 1$. Then for all $V$ open in $S'$ we have:

\begin{eqnarray*}
& \mathcal{H}om_{\pi_{S'}}\Big((u',u)^*\Delta_j,(u',u)^*(\mathcal{E}_1,\mathcal{V}_1,\phi_1)\otimes_{S'}
   \mathcal{M}\Big)(V)=&\\
& =\operatorname{Hom}_V\Big((u'^*\mathcal{P}_j,0,0)|_V,(u'^*\mathcal{E}_1\otimes_{C\times S'}
   \pi_{S'}^*\mathcal{M},u^*\mathcal{V}_1\otimes_{S'}\mathcal{M},\widetilde{\phi}_1)|_V\Big)= &\\
& =\operatorname{Hom}_V\Big((u'^*\mathcal{P}_j|_{\pi_{S'}^{-1}(V)},0,0),
   ((u'^*\mathcal{E}_1\otimes_{C\times S'}\pi_{S'}^*\mathcal{M})|_{\pi_{S'}^{-1}V},
   u^*\mathcal{V}_1\otimes_{S'}\mathcal{M}|_V,\widetilde{\phi}_1|_V)\Big)=&\\
& =\mathcal{H}om_{\pi_{S'}}\Big(u'^*\mathcal{P}_j,u'^*\mathcal{E}_1\otimes_{C\times S'}
   \pi_{S'}^*\mathcal{M}\Big)(V).&
\end{eqnarray*}

Therefore, we have:

\begin{eqnarray}
\nonumber &\mathcal{H}om_{\pi_{S'}}\Big((u',u)^*\Delta_j,(u',u)^*(\mathcal{E}_1,\mathcal{V}_1,\phi_1)\otimes_{S'}
   \mathcal{M}\Big)=&\\
\nonumber & =\mathcal{H}om_{\pi_{S'}}\Big(u'^*\mathcal{P}_j,u'^*\mathcal{E}_1\otimes_{C\times S'} 
   \pi_{S'}^*\mathcal{M}\Big)= &\\
\nonumber &=(\pi_{S'})_*\mathcal{H}om_{\mathcal{O}_{C\times S'}}\Big(u'^*\mathcal{P}_j,u'^*\mathcal{E}_1
   \otimes_{C\times S'}\pi_{S'}^*\mathcal{M}\Big)= &\\
\nonumber & =(\pi_{S'})_*\Big(u'^*\mathcal{P}_j^\vee\otimes_{C\times S'}u'^*\mathcal{E}_1
   \otimes_{C\times S'}\pi_{S'}^*\mathcal{M}\Big)= &\\
\nonumber &=(\pi_{S'})_*\Big(u'^*(\mathcal{P}_j^\vee\otimes_{C\times S}\mathcal{E}_1)
   \otimes_{C\times S'}\pi_{S'}^*\mathcal{M}\Big)= &\\
\label{22} & =(\pi_{S'})_*\Big(u'^*(\mathcal{P}_j^\vee\otimes_{C\times S}\mathcal{E}_1)\Big)
   \otimes_{S'}\mathcal{M}. &
\end{eqnarray}

Here the third equality is proved using the fact that $u'^*\mathcal{P}_j$ is locally free because $\mathcal{P}_j$
is so by proposition \ref{17} (iii). Analogous computations (with $u$ replaced by $\textrm{id}_S$ and
$\mathcal{M}$ by $\mathcal{O}_S$) prove that for all $j\geq 1$:

$$\mathcal{L}^j=\mathcal{H}om_{\pi_S}\Big(\Delta_j,(\mathcal{E}_1,\mathcal{V}_1,\phi_1)\Big)=
(\pi_S)_*\left(\mathcal{P}_j^\vee\otimes_{C\times S}\mathcal{E}_1\right).$$

Now by proposition \ref{17}  (v) we have that $\mathcal{L}^j$ is locally free on $S$; therefore by base change
(\cite[III, prop. 12.11 and prop. 12.5]{Ha}) we have:

$$(\pi_{S'})_*u'^*(\mathcal{P}_j^\vee\otimes_{C\times S}\mathcal{E}_1)=
u^*{\pi_S}_*(\mathcal{P}_j^\vee\otimes_{C\times S}\mathcal{E}_1)=u^*\mathcal{L}^j.$$

Therefore, we have that (\ref{22}) is equal to $u^*\mathcal{L}^j\otimes_{S'}\mathcal{M}$. So we have proved that
(\ref{21}) is true for all $j\geq 1$.\\

\textbf{Case (ii)} Let us suppose that $j=0$; then we have that $\Delta_0=(\mathcal{P}_0,0,0)\oplus(\pi_S^*
\mathcal{V}_2,\mathcal{V}_2,\textrm{id})$. By the same idea used in the previous case, we have a canonical isomorphism:

\begin{eqnarray*}
& \mathcal{H}om_{\pi_{S'}}\Big((u',u)^*(\mathcal{P}_0,0,0),(u',u)^*(\mathcal{E}_1,\mathcal{V}_1,\phi_1)
  \otimes_{S'}\mathcal{M}\Big)\simeq &\\
& \simeq u^*\mathcal{H}om_{\pi_{S}}\Big((\mathcal{P}_0,0,0),(\mathcal{E}_1,\mathcal{V}_1,\phi_1)\Big)
  \otimes_{S'}\mathcal{M}. &
\end{eqnarray*}

Therefore, in order to prove that (\ref{21}) is still valid for $j=0$, it suffices to prove
that there is a canonical isomorphism:

\begin{eqnarray}
\nonumber  & \mathcal{H}om_{\pi_{S'}}\Big((u',u)^*(\pi_S^*\mathcal{V}_2,\mathcal{V}_2,\textrm{id}),
   (u',u)^*(\mathcal{E}_1,\mathcal{V}_1,\phi_1)\otimes_{S'}\mathcal{M}\Big)\stackrel{?}{\simeq} &\\
\label{23} & \stackrel{?}{\simeq}u^*\mathcal{H}om_{\pi_{S}}\Big((\pi_S^*\mathcal{V}_2,\mathcal{V}_2,
   \textrm{id}),(\mathcal{E}_1,\mathcal{V}_1,\phi_1)\Big)\otimes_{S'}\mathcal{M}. &
\end{eqnarray}

Now for every open set $V$ in $S'$ we have:

\begin{eqnarray}
\nonumber &\mathcal{H}om_{\pi_{S'}}\Big((u',u)^*(\pi_S^*\mathcal{V}_2,\mathcal{V}_2,\textrm{id}),
   (u',u)^*(\mathcal{E}_1,\mathcal{V}_1,\phi_1)\otimes_{S'}\mathcal{M}\Big)(V)= &\\
\label{24} &=\operatorname{Hom}_V\Big((u'^*\pi_S^*\mathcal{V}_2,u^*\mathcal{V}_2,\widetilde{\textrm{id}})|_V,(u'^*
   \mathcal{E}_1\otimes_{C\times S'}\pi_{S'}^*\mathcal{M},u^*\mathcal{V}_1\otimes_{S'}\mathcal{M},
   \widetilde{\phi}_1)|_V\Big) &
\end{eqnarray}
where $\widetilde{\textrm{id}}$ is given by the composition:

$$\widetilde{\textrm{id}}:\,\pi_{S'}^*u^*\mathcal{V}_2\stackrel{\eta}{\longrightarrow}u'^*\pi_S^*\mathcal{V}_2
\stackrel{u'^*(\textrm{id})}{\longrightarrow}u'^*\pi_S^*\mathcal{V}_2$$
and $\eta$ is the canonical isomorphism induced by $\pi_S\circ u'=u\circ\pi_{S'}$ (see diagram (\ref{8})). Therefore,
$\widetilde{\textrm{id}}=\eta$ is an isomorphism. Hence (\ref{24}) is the set of all pairs $(\gamma,\delta)$ of the
form:

\begin{eqnarray*}
& \gamma:u'^*\pi_S^*\mathcal{V}_2|_{\pi_{S'}^{-1}(V)}\longrightarrow (u'^*\mathcal{E}_1\otimes_{C\times S'}
   \pi_{S'}^*\mathcal{M})|_{\pi_{S'}^{-1}(V)}, &\\
& \delta:u^*\mathcal{V}_2|_V\longrightarrow (u^*\mathcal{V}_1\otimes_{S'}\mathcal{M})|_V &
\end{eqnarray*}
such that they make this diagram commute:

\[\begin{tikzpicture}[xscale=3.1,yscale=-1.2]
    \node (A0_0) at (2, 0) {$u'^*\pi_S^*\mathcal{V}_2|_{\pi_{S'}^{-1}(V)}$};
    \node (A0_2) at (2, 2) {$(u'^*\mathcal{E}_1\otimes_{C\times S'}\pi_{S'}^*\mathcal{M})|_{\pi_{S'}^{-1}(V)}.$};
    \node (A1_1) at (1, 1) {$\curvearrowright$};
    \node (A2_0) at (0, 0) {$\pi_{S'}^*u^*\mathcal{V}_2|_{\pi_{S'}^{-1}(V)}$};
    \node (A2_2) at (0, 2) {$\pi_{S'}^*(u^*\mathcal{V}_1\otimes_{S'}\mathcal{M})|_{\pi_{S'}^{-1}V}$};
    \path (A0_0) edge [->]node [auto] {$\scriptstyle{\gamma}$} (A0_2);
    \path (A2_0) edge [->]node [auto] {$\scriptstyle{\eta|_{\pi_{S'}^{-1}(V)}}$} node [rotate=180,sloped]
       {$\scriptstyle{\widetilde{\ \ \ }}$} (A0_0);
    \path (A2_2) edge [->,swap]node [auto] {$\scriptstyle{{\widetilde{\phi}_1}|_{\pi_{S'}^{-1}(V)}}$} (A0_2);
    \path (A2_0) edge [->,swap]node [auto] {$\scriptstyle{\pi_{S'}^*\delta}$} (A2_2);
\end{tikzpicture}\]

Therefore, $\gamma$ is completely determined as

$$\gamma={\widetilde{\phi}_1}|_{\pi_{S'}^{-1}(V)}\circ(\pi_{S'}^*\delta)\circ\left(\eta|_{\pi_{S'}^{-1}(V)}
\right)^{-1}.$$

So, having fixed $(\pi_S^*\mathcal{V}_2,\mathcal{V}_2,\textrm{id})$, $(\mathcal{E}_1,\mathcal{V}_1,\phi_1)$, $u:S'
\rightarrow S$ and $\mathcal{M}$, we have that (\ref{24}) is naturally identified with the set of all morphisms
$\delta$ as before, i.e. with the set

$$\operatorname{Hom}_{V}(u^*\mathcal{V}_2|_V,(u^*\mathcal{V}_1\otimes_{S'}\mathcal{M})|_V)=
\mathcal{H}om_{\mathcal{O}_{S'}}(u^*\mathcal{V}_2,u^*\mathcal{V}_1\otimes_{S'}\mathcal{M})(V).$$

Therefore, the left hand side of (\ref{23}) is given by:

\begin{eqnarray*}
& \mathcal{H}om_{\mathcal{O}_{S'}}(u^*\mathcal{V}_2,u^*\mathcal{V}_1\otimes_{S'}\mathcal{M})=
   u^*\mathcal{V}_2^\vee\otimes_{S'}u^*\mathcal{V}_1\otimes_{S'}\mathcal{M}= &\\
& =u^*(\mathcal{V}_2^\vee\otimes_{S}\mathcal{V}_1)\otimes_{S'}\mathcal{M}=
   u^*\mathcal{H}om_{\mathcal{O}_{S}}(\mathcal{V}_2,\mathcal{V}_1)\otimes_{S'}\mathcal{M}. &
\end{eqnarray*}

Here we used several times the fact that $\mathcal{V}_2$ is locally free on $S$. By using the same idea
we can prove that also the right hand side of (\ref{23}) is given by the same expression,
so we conclude.
\end{proof}

With the same ideas we can also prove the following result; we omit the proof since it is quite similar to the previous
one.

\begin{lemma}\label{25}
For all $S$-schemes $u:S'\rightarrow S$, for all locally free $\mathcal{O}_{S'}$-modules $\mathcal{M}$ and for all
$j\geq 0$ there is a canonical isomorphism of sheaves on $S'$:

$$\mathcal{H}om_{\pi_{S'}}\Big((u',u)^*\Delta_j\otimes_{S'}\mathcal{M}^{\vee},(u',u)^*(\mathcal{E}_1,\mathcal{V}_1,\phi_1)\Big)
\simeq u^*\mathcal{L}^j\otimes_{S'}\mathcal{M}.$$
\end{lemma}

\begin{lemma}\label{26}
For every pair of families as before parametrized by $S$, for every morphism of schemes $u:S'\rightarrow S$, for every
locally free $\mathcal{O}_{S'}$-module $\mathcal{M}$ and for every $i\geq 0$ we have a canonical isomorphism of
sheaves over $S'$:

\begin{eqnarray*}
& \mathcal{E}xt^i_{\pi_{S'}}\Big((u',u)^*(\mathcal{E}_2,\mathcal{V}_2,\phi_2),(u',u)^*
   (\mathcal{E}_1,\mathcal{V}_1,\phi_1)\otimes_{S'}\mathcal{M}\Big)= &\\
& =\mathcal{H}^i\Big(\mathcal{H}om_{\pi_{S'}}\Big((u',u)^*\Delta_{\bullet},
    (u',u)^*(\mathcal{E}_1,\mathcal{V}_1,\phi_1)\otimes_{S'}\mathcal{M}\Big)\Big) &
\end{eqnarray*}
where $\Delta_{\bullet}$ is as in proposition \ref{17}.
\end{lemma}

\begin{proof}
By remark \ref{18} we get that $(u',u)^*\Delta_{\bullet}$ is a very negative resolution of $(u',u)^*$ $(\mathcal{E}_2,
\mathcal{V}_2,\phi_2)$ with respect to $(u',u)^*(\mathcal{E}_1,\mathcal{V}_1,\phi_1)$ $\otimes_{S'}\mathcal{M}$ for all
locally free $\mathcal{O}_{S'}$-modules $\mathcal{M}$. Therefore we can use lemma \ref{19} for the families
$(u',u)^*(\mathcal{E}_2,\mathcal{V}_2,\phi_2)$ and $(u',u)^*$ $(\mathcal{E}_1,\mathcal{V}_1,\phi_1)$ $\otimes_{S'}
\mathcal{M}$ over $S'$ and we conclude.
\end{proof}

Now if we combine lemma \ref{26} with the canonical isomorphism of lemma \ref{20}, we get the following statement,
that is analogous to  \cite[cor. 1.2.iii]{L}.

\begin{lemma}
For every $i\geq 0$, for every morphism $u:S'\rightarrow S$ and for every locally free $\mathcal{O}_{S'}$-module
$\mathcal{M}$ we have a canonical isomorphism of sheaves over $S'$:

$$\mathcal{E}xt^i_{\pi_{S'}}\Big((u',u)^*(\mathcal{E}_2,\mathcal{V}_2,\phi_2),(u',u)^*(\mathcal{E}_1,\mathcal{V}_1,
\phi_1)\otimes_{S'}\mathcal{M}\Big)=\mathcal{H}^i(u^*\mathcal{L}^{\bullet}\otimes_{S'}\mathcal{M})$$
where $\mathcal{L}^{\bullet}$ is as in proposition \ref{17}. Since $\mathcal{L}^{\bullet}$ is a complex of locally
free sheaves on $S$ by that proposition, this implies that the sheaf on the left is coherent on $S'$.
\end{lemma}

Now for every locally free $\mathcal{O}_S$-module $\mathcal{M}$ and for every $i\geq 0$, we define

$$\mathcal{T}^i(\mathcal{M}):=\mathcal{H}^i(\mathcal{L}^{\bullet}\otimes_S\mathcal{M})=
\mathcal{E}xt^i_{\pi_S}\Big((\mathcal{E}_2,\mathcal{V}_2,\phi_2),(\mathcal{E}_1,\mathcal{V}_1,\phi_1)\otimes_S
\mathcal{M}\Big),$$
where the last identity is given by the previous lemma with $u=id_S$. By \cite[III, proposition 12.5]{Ha} we get
natural homomorphisms for every $i\geq 0$:

$$\varphi(i,\mathcal{M}):\mathcal{T}^i(\mathcal{O}_S)\otimes_S\mathcal{M}\rightarrow\mathcal{T}^i(\mathcal{M}).$$

Moreover, for every morphism $u: S'\rightarrow S$ by using the same computation as
\cite[III, proposition 9.3 and remark 9.3.1]{Ha} we get the base change homomorphism:

$$\tau^i(u):u^*\mathcal{E}xt^i_{\pi_S}\Big((\mathcal{E}_2,\mathcal{V}_2,\phi_2),(\mathcal{E}_1,\mathcal{V}_1,\phi_1)
\Big)\rightarrow$$
$$\rightarrow\mathcal{E}xt^i_{\pi_{S'}}\Big((u',u)^*(\mathcal{E}_2,\mathcal{V}_2,\phi_2),(u',u)^*
(\mathcal{E}_1,\mathcal{V}_1,\phi_1)\Big).$$

In addition, by using again a resolution $\Delta_{\bullet}$ as in proposition \ref{17} together with
\cite[III, proposition 9.3]{Ha}, we get the following result.

\begin{proposition}
For every \emph{flat} morphism $u:S'\rightarrow S$ of schemes and for every $i\geq 0$, the base change
homomorphism $\tau^i(u)$ is an isomorphism.
\end{proposition}

This is exactly \cite[th\'eor\`eme 1.16 (i)]{He}, but with a more explicit construction of such an isomorphism,
that was not described in that work. Moreover, by proceeding as in \cite[III.12]{Ha} we get the following result.

\begin{proposition}\label{27}
\emph{(cohomology and base change for families of coherent systems)} Let $S$ be any scheme and let
$(\mathcal{E}_l,\mathcal{V}_l,\phi_l)$ be two families of coherent systems parametrized by $S$ for $l=1,2$. Let $s$ be
any point in $S$ and let us assume that the base change homomorphism 

$$\tau^i(s):\mathcal{E}xt^i_{\pi_S}\Big((\mathcal{E}_2,\mathcal{V}_2,\phi_2),(\mathcal{E}_1,\mathcal{V}_1,\phi_1)\Big)
\otimes k(s) \rightarrow \operatorname{Ext}^i\Big((\mathcal{E}_2,\mathcal{V}_2,\phi_2)_s,(\mathcal{E}_1,\mathcal{V}_1,
\phi_1)_s\Big)$$
is surjective. Then:

\begin{enumerate}[(i)]
 \item there is an open neighbourhood $U$ of $s$ in $S$ such that $\tau^i(s')$ is an isomorphism for all $s'$ in $U$;
 \item $\tau^{i-1}(s)$ is surjective if and only if $\mathcal{E}xt^i_{\pi_S}\Big((\mathcal{E}_2,\mathcal{V}_2,\phi_2),
   (\mathcal{E}_1,\mathcal{V}_1,\phi_1)\Big)$ is locally free in an open neighbourhood of $s$ in $S$.
\end{enumerate}
\end{proposition}

According to the usual definitions for coherent sheaves, if $\tau^i(s)$ is an isomorphism for all $s$ in $S$, then we
will say that $\mathcal{E}xt^i_{\pi_S}\Big((\mathcal{E}_2,\mathcal{V}_2,\phi_2),(\mathcal{E}_1,\mathcal{V}_1,\phi_1)
\Big)$  \emph{commutes with base change}. If this is the case, then $\tau^i(u)$ is an isomorphism for all morphisms
$u:S'\rightarrow S$ of schemes.

\begin{remark}
From now on, we will not need to refer explicitly to the maps of the form $\phi$, so in the following lemmas and
propositions we will use the notation of definition \ref{3} for families of coherent systems.
\end{remark}

Exactly as in \cite[lemma 4.1]{L}, we can prove the following consequence of lemmas \ref{20} and \ref{25}.

\begin{lemma}\label{28}
For every scheme $S$, for every pair of families $(\mathcal{E}_l,\mathcal{V}_l)$ parametrized by $S$ for $l=1,2$, for
every locally free $\mathcal{O}_S$-module $\mathcal{M}$ and for every $i\geq 0$, there are canonical isomorphisms

$$\mathcal{E}xt^i_{\pi_S}\Big((\mathcal{E}_2,\mathcal{V}_2),(\mathcal{E}_1,\mathcal{V}_1)\Big)\otimes_S\mathcal{M}
\simeq\mathcal{E}xt^i_{\pi_S}\Big((\mathcal{E}_2,\mathcal{V}_2),(\mathcal{E}_1,\mathcal{V}_1)\otimes_S\mathcal{M}
\Big)\simeq$$
$$\simeq\mathcal{E}xt^i_{\pi_S}\Big((\mathcal{E}_2,\mathcal{V}_2)\otimes_S\mathcal{M}^{\vee},(\mathcal{E}_1,\mathcal{V}_1)
\Big).$$
\end{lemma}

\begin{lemma}\label{29}
Let us fix any scheme $S$, any pair of families $(\mathcal{E}_l,\mathcal{V}_l)$ parametrized by $S$ for $l=1,2$, any
$S$-scheme $u:S'\rightarrow S$ and any locally free $\mathcal{O}_{S'}$-module $\mathcal{M}$. Then there is a canonical
morphism:

\begin{eqnarray*}
& \mu: \operatorname{Ext}^1_{S'}\Big((u',u)^*(\mathcal{E}_2,\mathcal{V}_2),(u',u)^*(\mathcal{E}_1,\mathcal{V}_1)
   \otimes_{S'}\mathcal{M}\Big)\rightarrow &\\
& \rightarrow \operatorname{H}^0\Big(S',\mathcal{E}xt^1_{\pi_{S'}}\Big((u',u)^*(\mathcal{E}_2,
   \mathcal{V}_2),(u',u)^*(\mathcal{E}_1,\mathcal{V}_1)\Big)\otimes_{S'}\mathcal{M}\Big). &
\end{eqnarray*}

Let us assume that one of the following 2 conditions hold:

\begin{enumerate}[(a)]
 \item $\operatorname{Hom}((\mathcal{E}_2,\mathcal{V}_2)_s,(\mathcal{E}_1,\mathcal{V}_1)_s)=0$ for all $s$ in $S$;
 \item $S'$ is affine.
\end{enumerate}

Then $\mu$ is an isomorphism.
\end{lemma}

\begin{proof}
We recall that by \cite[proposition A.9]{BGMMN}, there is a spectral sequence

\begin{eqnarray*}
& \operatorname{H}^p\Big(S',\mathcal{E}xt^q_{\pi_{S'}}\Big((u',u)^*(\mathcal{E}_2,\mathcal{V}_2),(u',u)^*
  (\mathcal{E}_1,\mathcal{V}_1)\otimes_{S'}\mathcal{M}\Big)\Big)\Rightarrow &\\
& \Rightarrow\operatorname{Ext}^{p+q}_{S'}\Big((u',u)^*(\mathcal{E}_2,\mathcal{V}_2),(u',u)^*(\mathcal{E}_1,
   \mathcal{V}_1)\otimes_{S'}\mathcal{M}\Big). &
\end{eqnarray*}

Moreover,  by lemma \ref{28} (over $S'$ instead of $S$) we have for every $q\geq 0$:

\begin{eqnarray*}
& \mathcal{E}xt^q_{\pi_{S'}}\Big((u',u)^*(\mathcal{E}_2,\mathcal{V}_2),(u',u)^*(\mathcal{E}_1,\mathcal{V}_1)
   \otimes_{S'}\mathcal{M}\Big)= &\\
& =\mathcal{E}xt^q_{\pi_{S'}}\Big((u',u)^*(\mathcal{E}_2,\mathcal{V}_2),(u',u)^*(\mathcal{E}_1,\mathcal{V}_1)
   \Big)\otimes_{S'}\mathcal{M}.
\end{eqnarray*}

So we get a long exact sequence:

\begin{eqnarray}
\nonumber & 0\rightarrow \operatorname{H}^1\Big(S',\mathcal{H}om_{\pi_{S'}}\Big((u',u)^*(\mathcal{E}_2,\mathcal{V}_2),
   (u',u)^*(\mathcal{E}_1,\mathcal{V}_1)\Big)\otimes_{S'}\mathcal{M}\Big)\rightarrow &\\
\nonumber & \rightarrow\operatorname{Ext}^1_{S'}\Big((u',u)^*(\mathcal{E}_2,\mathcal{V}_2),
   (u',u)^*(\mathcal{E}_1,\mathcal{V}_1)\otimes_{S'}\mathcal{M}\Big)\stackrel{\mu}{\rightarrow} &\\
\nonumber &\stackrel{\mu}{\rightarrow} \operatorname{H}^0\Big(S',\mathcal{E}xt^1_{\pi_{S'}}\Big((u',u)^*(\mathcal{E}_2,
   \mathcal{V}_2),(u',u)^*(\mathcal{E}_1,\mathcal{V}_1)\Big)\otimes_{S'}\mathcal{M}\Big)\rightarrow &\\
\label{30} & \rightarrow \operatorname{H}^2\Big(S',\mathcal{H}om_{\pi_{S'}}\Big((u',u)^*(\mathcal{E}_2,\mathcal{V}_2),
   (u',u)^*(\mathcal{E}_1,\mathcal{V}_1)\Big)\otimes_{S'}\mathcal{M}\Big)\rightarrow\cdots &
\end{eqnarray}

Now let us assume (a): this implies that the base change morphisms $\tau^0(s)$ are surjective for all $s$ in $S$.
Therefore, by cohomology and base change

$$\mathcal{H}om_{\pi_{S'}}((u',u)^*\Big(\mathcal{E}_2,\mathcal{V}_2),(u',u)^*(\mathcal{E}_1,\mathcal{V}_1)\Big)=0.$$

So by substituting in the previous long exact sequence we that $\mu$ is an isomorphism. If we assume (b) then both
the first and the last term of (\ref{30}) are zero, so we conclude as before.
\end{proof}

\section{Families of (classes of) extensions}

The following lemma is a direct consequence of the definition of the functors $\operatorname{Ext}^1_S(-,-)$'s and it is
already implicit in the proof of \cite[proposition A.9]{BGMN}. The lemma is also stated explicitly in
\cite[proposition 3.1]{BGMN} in the particular case when $S=\textrm{Spec}(\mathbb{C})$.

\begin{lemma}\label{32}
For all schemes $S$ and for all pairs of families $(\mathcal{E}_l,\mathcal{V}_l)$ parametrized by $S$ for $l=1,2$,
there is a canonical bijection from $\operatorname{Ext}^1_S((\mathcal{E}_2,\mathcal{V}_2), (\mathcal{E}_1,
\mathcal{V}_1))$ to the set of all short exact sequences

\begin{equation}\label{33}
0\longrightarrow(\mathcal{E}_1,\mathcal{V}_1)\longrightarrow(\mathcal{E},\mathcal{V})
\longrightarrow(\mathcal{E}_2,\mathcal{V}_2)\longrightarrow 0
\end{equation}
modulo equivalences.
\end{lemma}

Here an extension (\ref{33}) is equivalent to an extension

$$0\longrightarrow(\mathcal{E}_1,\mathcal{V}_1)\longrightarrow(\mathcal{E}',\mathcal{V}')\longrightarrow
(\mathcal{E}_2,\mathcal{V}_2)\longrightarrow 0$$
if and only if there is an isomorphism $(\mathcal{E},\mathcal{V})\stackrel{\sim}{\rightarrow}(\mathcal{E}',
\mathcal{V}')$ making the following diagram commute

\[\begin{tikzpicture}[xscale=1.5,yscale=-1.2]
    \node (A0_0) at (0, 0) {$0$};
    \node (A0_2) at (2, 0) {$(\mathcal{E}_1,\mathcal{V}_1)$};
    \node (A0_4) at (4, 0) {$(\mathcal{E},\mathcal{V})$};
    \node (A0_6) at (6, 0) {$(\mathcal{E}_2,\mathcal{V}_2)$};
    \node (A0_8) at (8, 0) {$0$};
    \node (A1_3) at (3, 1) {$\curvearrowright$};
    \node (A2_0) at (0, 2) {$0$};
    \node (A1_5) at (5, 1) {$\curvearrowright$};
    \node (A2_2) at (2, 2) {$(\mathcal{E}_1,\mathcal{V}_1)$};
    \node (A2_4) at (4, 2) {$(\mathcal{E}',\mathcal{V}')$};
    \node (A2_6) at (6, 2) {$(\mathcal{E}_2,\mathcal{V}_2)$};
    \node (A2_8) at (8, 2) {$0.$};
    \path (A0_0) edge [->]node [auto] {$\scriptstyle{}$} (A0_2);
    \path (A0_4) edge [->]node [auto] {$\scriptstyle{}$} node [rotate=180,sloped]
       {$\scriptstyle{\widetilde{\ \ \ }}$} (A2_4);
    \path (A0_2) edge [-,double distance=1.5pt]node [auto] {$\scriptstyle{}$} (A2_2);
    \path (A0_6) edge [-,double distance=1.5pt]node [auto] {$\scriptstyle{}$} (A2_6);
    \path (A0_4) edge [->]node [auto] {$\scriptstyle{}$} (A0_6);
    \path (A2_2) edge [->]node [auto] {$\scriptstyle{}$} (A2_4);
    \path (A0_2) edge [->]node [auto] {$\scriptstyle{}$} (A0_4);
    \path (A2_6) edge [->]node [auto] {$\scriptstyle{}$} (A2_8);
    \path (A2_0) edge [->]node [auto] {$\scriptstyle{}$} (A2_2);
    \path (A2_4) edge [->]node [auto] {$\scriptstyle{}$} (A2_6);
    \path (A0_6) edge [->]node [auto] {$\scriptstyle{}$} (A0_8);
\end{tikzpicture}\]

\begin{proof}
This is a standard fact for an abelian category with enough injectives. It is therefore sufficient to observe that,
given a short exact sequence in the category of families of weak coherent systems on $C\times S/S$ for which the left
and right hand members are families of coherent systems parametrized by $S$, then the whole sequence belongs to the
category of coherent systems parametrized by $S$.
\end{proof}

Let us consider any scheme $S$ and any pair of families parametrized by $S$ as before and an extension like (\ref{33}).
For every point $s$ in $S$ the pullback of such an exact sequence to $C_s=C\times\{s\}$ gives rise to an extension:

$$0\longrightarrow(\mathcal{E}_1,\mathcal{V}_1)_s\longrightarrow(\mathcal{E},\mathcal{V})_s\longrightarrow
(\mathcal{E}_2,\mathcal{V}_2)_s\longrightarrow 0.$$

Therefore, by lemma \ref{32} we get a well defined linear map:

$$\Phi_s:\operatorname{Ext}^1_S\Big((\mathcal{E}_2,\mathcal{V}_2),(\mathcal{E}_1,\mathcal{V}_1)\Big)\rightarrow
\operatorname{Ext}^1\Big((\mathcal{E}_2,\mathcal{V}_2)_s,(\mathcal{E}_1,\mathcal{V}_1)_s\Big).$$

As in \cite{L}, we give the definition of family of extensions as follows:

\begin{defin}
A \emph{family of (classes of) extensions of }$(\mathcal{E}_2,\mathcal{V}_2)$\emph{ by }$(\mathcal{E}_1,\mathcal{V}_1)$
\emph{over }$S$ is any family 

$$\Big\{e_s\in\operatorname{Ext}^1\Big((\mathcal{E}_2,\mathcal{V}_2)_s,(\mathcal{E}_1,\mathcal{V}_1)_s\Big)
\Big\}_{s\in S}$$
such that there is an open covering $\mathfrak{U}=\{U_i\}_{i\in I}$ of $S$ and for each $i\in I$ there is an element
$\sigma_i$ in $\operatorname{Ext}^1_{U_i}\Big((\mathcal{E}_2,\mathcal{V}_2)|_{U_i},(\mathcal{E}_1,
\mathcal{V}_1)|_{U_i}\Big)$ such that $e_s=\Phi_{i,s}(\sigma_i)$ for every $s$ in $S$ and for every $i\in I$ such that
$s\in U_i$. Here $\Phi_{i,s}$ denotes the linear map

$$\Phi_{i,s}:\operatorname{Ext}^1_{U_i}\Big((\mathcal{E}_2,\mathcal{V}_2)|_{U_i},(\mathcal{E}_1,\mathcal{V}_1)|_{U_i}
\Big)\rightarrow\operatorname{Ext}^1\Big((\mathcal{E}_2,\mathcal{V}_2)_s,(\mathcal{E}_1,\mathcal{V}_1)_s\Big).$$

A family of extensions is called \emph{globally defined} if the covering $\mathfrak{U}$ can be chosen to coincide
with $\{S\}$.
\end{defin}

For every $s$ in $S$, let us define the canonical homomorphism 

$$\iota_s:\mathcal{E}xt^1_{\pi_S}\Big((\mathcal{E}_2,\mathcal{V}_2),(\mathcal{E}_1,\mathcal{V}_1)\Big)\rightarrow
\mathcal{E}xt^1_{\pi_S}\Big((\mathcal{E}_2,\mathcal{V}_2),(\mathcal{E}_1,\mathcal{V}_1)\Big)\otimes k(s).$$

Then we get a result analogous to that of \cite[lemma 2.1]{L}.

\begin{lemma}\label{31}
For every $s$ in $S$, the map $\Phi_s$ coincides with the composition:

\begin{eqnarray*}
& \operatorname{Ext}^1_S\Big((\mathcal{E}_2,\mathcal{V}_2),(\mathcal{E}_1,\mathcal{V}_1)\Big)
   \stackrel{\mu}{\longrightarrow}\operatorname{H}^0\Big(S,\mathcal{E}xt^1_{\pi_S}\Big((\mathcal{E}_2,\mathcal{V}_2),
   (\mathcal{E}_1,\mathcal{V}_1)\Big)\Big)\stackrel{\operatorname{H}^0(S,\iota_s)}{\longrightarrow} &\\
& \stackrel{\operatorname{H}^0(S,\iota_s)}{\longrightarrow} \operatorname{H}^0\Big(S,\mathcal{E}xt^1_{\pi_S}\Big(
   (\mathcal{E}_2,\mathcal{V}_2),(\mathcal{E}_2,\mathcal{V}_2)\Big)\otimes k(s)\Big)= &\\
& =\mathcal{E}xt^1_{\pi_S}\Big((\mathcal{E}_2,\mathcal{V}_2),(\mathcal{E}_1,\mathcal{V}_1)\Big)\otimes k(s)
   \stackrel{\tau^1(s)}{\longrightarrow} \operatorname{Ext}^1\Big((\mathcal{E}_2,\mathcal{V}_2)_s,(\mathcal{E}_1,
   \mathcal{V}_1)_s \Big), &
\end{eqnarray*}
where $\tau^1(s)$ is the base change homomorphism induced by the inclusion of $s$ in $S$ and $\mu$ is the map
described in lemma \ref{29} with $u=\textrm{id}_S$ and $\mathcal{M}=\mathcal{O}_S$ ($\mu$ is not necessarily an isomorphism
in this case).
\end{lemma}

\begin{defin}
Having fixed any pair of families $(\mathcal{E}_l,\mathcal{V}_l)$ parametrized by a scheme $S$ for $l=1,2$, we define

$$\operatorname{EXT}\Big((\mathcal{E}_2,\mathcal{V}_2),(\mathcal{E}_1,\mathcal{V}_1)\Big)$$
as the set of all the families of extensions between these two families of coherent systems; we consider also
the subset

$$\operatorname{EXT}_{glob}\Big((\mathcal{E}_2,\mathcal{V}_2),(\mathcal{E}_1,\mathcal{V}_1)\Big)$$
consisting of those families of extensions that are globally defined. Both sets have a natural structure of vector
space.
\end{defin}

\begin{proposition}\label{34}
Let us suppose that $S$ is \emph{reduced} and that $\mathcal{E}xt^1_{\pi_S}((\mathcal{E}_2,\mathcal{V}_2),
(\mathcal{E}_1,\mathcal{V}_1))$ commutes with base change. Then there is a canonical isomorphism between the space
$\operatorname{EXT}_{glob}((\mathcal{E}_2,\mathcal{V}_2),(\mathcal{E}_1,\mathcal{V}_1))$ and 

$$\operatorname{Ext}^1_S\Big((\mathcal{E}_2,\mathcal{V}_2),(\mathcal{E}_1,\mathcal{V}_1)\Big)/\operatorname{H}^1\Big(S,
\mathcal{H}om_{\pi_S}\Big((\mathcal{E}_2,\mathcal{V}_2),(\mathcal{E}_1,\mathcal{V}_1)\Big)\Big)\subseteq$$
$$\subseteq \operatorname{H}^0\Big(S,\mathcal{E}xt^1_{\pi_S}\Big((\mathcal{E}_2,\mathcal{V}_2),(\mathcal{E}_1,
\mathcal{V}_1)\Big)\Big).$$
\end{proposition}

\begin{proof}
For every class of extensions $\sigma\in\operatorname{Ext}^1_S((\mathcal{E}_2,\mathcal{V}_2),(\mathcal{E}_1,
\mathcal{V}_1))$, by lemma \ref{31} the family

$$\{\Phi_s(\sigma)=(\tau^1(s)\circ \operatorname{H}^0(S,\iota_s)\circ\mu)(\sigma)\}_{s\in S}$$
is a globally defined family of extensions of $(\mathcal{E}_2,\mathcal{V}_2)$ by $(\mathcal{E}_1,\mathcal{V}_1)$ over
$S$. Let us consider the exact sequence (\ref{30}) with $u=\textrm{id}_S$ and $\mathcal{M}=\mathcal{O}_S$ and let
us denote by $\overline{\mu}$ the morphism induced by $\mu$:

$$H:=\operatorname{Ext}^1_S\Big((\mathcal{E}_2,\mathcal{V}_2),(\mathcal{E}_1,\mathcal{V}_1)\Big)/\operatorname{H}^1
\Big(S,\mathcal{H}om_{\pi_S}\Big((\mathcal{E}_2,\mathcal{V}_2),(\mathcal{E}_1,\mathcal{V}_1)\Big)\Big)
\stackrel{\overline{\mu}}{\rightarrow}$$
$$\stackrel{\overline{\mu}}{\rightarrow}
\operatorname{H}^0\Big(S,\mathcal{E}xt^1_{\pi_S}\Big((\mathcal{E}_2,\mathcal{V}_2),(\mathcal{E}_1,\mathcal{V}_1)\Big)
\Big).$$

Now let us consider the set map $f$ defined from $H$ to $\operatorname{EXT}_{glob}$ as follows: for every class
$[\sigma]$ in $H$ we associate to it the family

$$f([\sigma]):=\{(\tau^1(s)\circ \operatorname{H}^0(S,\iota_s)\circ\bar{\mu})([\sigma])\}_{s\in S}=\{\Phi_s(\sigma)
\}_{s\in S}.$$

Now $\overline{\mu}$ is injective by construction and by (\ref{30}). Moreover the family $\{\iota_s\}_{s\in S}$ is
injective by using Nakayama's lemma and the fact that $S$ is reduced by hypothesis. So also the family
$\{\operatorname{H}^0(S,\iota_s)\}_{s\in S}$ is injective. In addition, every $\tau^1(s)$ is an isomorphism by hypothesis
(base change for $i=1$), so in particular it is injective. Therefore the set map $f$ is injective. Moreover, $f$ is
surjective by definition of globally defined family and by lemma \ref{31}. Finally, this map is clearly linear, so we
get the desired isomorphism.
\end{proof}

\begin{proposition}\label{35}
Let us assume the same hypotheses as for proposition \ref{34}. Then there is a canonical isomorphism

$$\operatorname{EXT}\Big((\mathcal{E}_2,\mathcal{V}_2),(\mathcal{E}_1,\mathcal{V}_1)\Big)\simeq\operatorname{H}^0
\Big(S,\mathcal{E}xt^1_{\pi_S}\Big((\mathcal{E}_2,\mathcal{V}_2),(\mathcal{E}_1,\mathcal{V}_1)\Big)\Big).$$
\end{proposition}

\begin{proof}
Let us fix any $\sigma\in \operatorname{H}^0(S,\mathcal{E}xt^1_{\pi_S}((\mathcal{E}_2,\mathcal{V}_2),$ $(\mathcal{E}_1,
\mathcal{V}_1)))$, let $\mathfrak{U}=\{U_i\}_{i\in I}$ be any open \emph{affine} covering of $S$ and let $\sigma_i:=
\sigma|_{U_i}$. By lemma \ref{29} (b) for $u:U_i\hookrightarrow S$ and $\mathcal{M}=\mathcal{O}_{U_i}$, for all
$i\in I$ we have an isomorphism

\begin{equation}\label{36}
\mu_i:\operatorname{Ext}^1_{U_i}\Big((\mathcal{E}_2,\mathcal{V}_2)|_{U_i},(\mathcal{E}_1,\mathcal{V}_1)|_{U_i}\Big)
\stackrel{\sim}{\rightarrow}\operatorname{H}^0\Big(U_i,\mathcal{E}xt^1_{\pi_S}\Big((\mathcal{E}_2,\mathcal{V}_2)|_{U_i},
(\mathcal{E}_1,\mathcal{V}_1)|_{U_i}\Big)\Big).
\end{equation}

For every point $s\in U_i$, we define $e_s:=\Phi_{i,s}(\mu_i^{-1}(\sigma_i))$; a direct check proves that such an
extension is well defined, i.e. it depends only on $s$ and not on $i$. So the family $\{e_s\}_{s\in S}$ is a family of
extensions of $(\mathcal{E}_2,\mathcal{V}_2)$ by $(\mathcal{E}_1,\mathcal{V}_1)$ over $S$. Since $\sigma$ is a global
section of $\mathcal{E}xt^1_{\pi_S}((\mathcal{E}_2,\mathcal{V}_2),(\mathcal{E}_1,\mathcal{V}_1))$, a direct computation
shows that such a family does not depend on the choice of the affine covering $\mathfrak{U}$. So we get a well defined
linear map

\begin{equation}\label{37}
\operatorname{H}^0\Big(S,\mathcal{E}xt^1_{\pi_S}\Big((\mathcal{E}_1,\mathcal{V}_1),(\mathcal{E}_2,\mathcal{V}_2)\Big)
\Big)\longrightarrow\operatorname{EXT}\Big((\mathcal{E}_2,\mathcal{V}_2),(\mathcal{E}_1,\mathcal{V}_1)\Big).
\end{equation}

We explicitly describe an inverse for such a map. Let $\{e_s\}_{s\in S}$ be any family in the set $\operatorname{EXT}
(-,-)$. By definition of family of extensions, there is an open covering $\mathfrak{U}=\{U_i\}_{i\in I}$ of $S$ and for
every $i$ there is an object 

$$\widetilde{\sigma_i}\in\operatorname{Ext}^1_{U_i}\Big((\mathcal{E}_2,\mathcal{V}_2)|_{U_i},(\mathcal{E}_1,
\mathcal{V}_1)|_{U_i}\Big)$$
such that $e_s=\Phi_{i,s}(\widetilde{\sigma_i})$ for all $s\in U_i$. Without loss of generality, we can assume that
$\mathfrak{U}$ is an affine covering. Therefore we can use (\ref{36}) and we define

$$\sigma_i:=\mu_i(\widetilde{\sigma_i})\in \operatorname{H}^0\Big(U_i,\mathcal{E}xt^1_{\pi_S}\Big((\mathcal{E}_2,\mathcal{V}_2),
(\mathcal{E}_1,\mathcal{V}_1)\Big)\Big).$$

As in lemma \ref{31} on $U_i$ instead of $S$, we get that for every $i\in I$ and for every $s$ in $U_i$, the morphism
$\Phi_{i,s}$ coincides with the composition:

\begin{eqnarray*}
& \operatorname{Ext}^1_{U_i}\Big((\mathcal{E}_2,\mathcal{V}_2)|_{U_i},(\mathcal{E}_1,\mathcal{V}_1)|_{U_i}\Big)
   \stackrel{\mu_i}{\longrightarrow}\operatorname{H}^0\Big(U_i,\mathcal{E}xt^1_{\pi_S}\Big((\mathcal{E}_2,
   \mathcal{V}_2),(\mathcal{E}_1,\mathcal{V}_1)\Big)\Big)\stackrel{\operatorname{H}^0(U_i,\iota_s)}{\longrightarrow} &\\
& \stackrel{\operatorname{H}^0(U_i,\iota_s)}{\longrightarrow} \operatorname{H}^0\Big(U_i,\mathcal{E}xt^1_{\pi_S}\Big(
  (\mathcal{E}_2,\mathcal{V}_2),(\mathcal{E}_1,\mathcal{V}_1)\Big)\otimes k(s)\Big)= &\\
& =\mathcal{E}xt^1_{\pi_S}\Big((\mathcal{E}_2,\mathcal{V}_2),(\mathcal{E}_1,\mathcal{V}_1)\Big)\otimes k(s)
   \stackrel{\tau^1(s)}{\longrightarrow} \operatorname{Ext}^1\Big((\mathcal{E}_2,\mathcal{V}_2)_s,
   (\mathcal{E}_1,\mathcal{V}_1)_s\Big). &
\end{eqnarray*}

So for every $s\in U_i$ we have:

$$\Phi_{i,s}(\widetilde{\sigma_i})=(\tau^1(s)\circ \operatorname{H}^0(U_i,\iota_s)\circ\mu_i)(\widetilde{\sigma_i})=
(\tau^1(s)\circ\operatorname{H}^0(U_i,\iota_s))(\sigma_i)=\tau^1(s)(\sigma_{i,s}).$$

Analogously, for every $s\in U_j$ we have $\Phi_{j,s}(\widetilde{\sigma_j})=\tau^1(s)(\sigma_{j,s})$. By definition
of family of extensions, if $s\in U_i\cap U_j$ then 

$$\tau^1(s)(\sigma_{i,s})=\tau^1(s)(\sigma_{j,s}).$$

By hypothesis, $\tau^1(s)$ is an isomorphism for all $s$ in $S$, so we conclude that for all pairs $i,j$ in $I$ and for
all $s\in U_i\cap U_j$ we have $\sigma_{i,s}=\sigma_{j,s}$. Since $S$ is reduced, we conclude that $\sigma_i$ coincides
with $\sigma_j$ over $U_i\cap U_j$. So there exists a unique

$$\sigma\in \operatorname{H}^0\Big(S,\mathcal{E}xt^1_{\pi_S}\Big((\mathcal{E}_2,\mathcal{V}_2),(\mathcal{E}_1,\mathcal{V}_1)
\Big)\Big)$$
such that $\sigma|_{U_i}=\sigma_i$ for all $i\in I$. A direct computation shows that $\sigma$ does not depend on the
choice of the covering $\mathfrak{U}$ nor on the choice of the family $\{\widetilde{\sigma_i}\}_{i\in I}$, so we get
a well defined map

\begin{equation}\label{38}
\operatorname{EXT}\Big((\mathcal{E}_2,\mathcal{V}_2),(\mathcal{E}_1,\mathcal{V}_1)\Big)\rightarrow\operatorname{H}^0
\Big(S,\mathcal{E}xt^1_{\pi_S}\Big((\mathcal{E}_2,\mathcal{V}_2),(\mathcal{E}_1,\mathcal{V}_1)\Big)\Big).
\end{equation}

Now it is easy to see that the map in (\ref{38}) is the inverse of (\ref{37}), so we conclude.
\end{proof}

\section{Universal families of extensions}
Now let us suppose that $\mathcal{E}xt^1_{\pi_S}((\mathcal{E}_2,\mathcal{V}_2),(\mathcal{E}_1,\mathcal{V}_1))$ 
\emph{commutes with base change}. Then let us define a contravariant functor $F$ from the category of $S$-schemes to
the category of sets. For every morphism $u:S'\rightarrow S$, let us consider the pullback diagram (\ref{8}) and let
us define:

$$F(S'):=\operatorname{H}^0\Big(S',\mathcal{E}xt^1_{\pi_{S'}}\Big((u',u)^*(\mathcal{E}_2,\mathcal{V}_2),
(u',u)^*(\mathcal{E}_1,\mathcal{V}_1)\Big)\Big).$$

For every morphism $v: S''\rightarrow S'$ of $S$-schemes we define $F(v):F(S')\rightarrow F(S'')$ as the
composition:

\begin{eqnarray*}
& \operatorname{H}^0\Big(S',\mathcal{E}xt^1_{\pi_{S'}}\Big((u',u)^*(\mathcal{E}_2,
   \mathcal{V}_2),(u',u)^*(\mathcal{E}_1,\mathcal{V}_1)\Big)\Big)\longrightarrow &\\
& \longrightarrow \operatorname{H}^0\Big(S'',v^*\mathcal{E}xt^1_{\pi_{S'}}\Big((u',u)^*(\mathcal{E}_2,\mathcal{V}_2),
   (u',u)^*(\mathcal{E}_1,\mathcal{V}_1)\Big)\Big)\stackrel{\operatorname{H}^0(S'',\tau^1(v))}{\longrightarrow} &\\
& \stackrel{\operatorname{H}^0(S'',\tau^1(v))}{\longrightarrow}\operatorname{H}^0\Big(S'',
   \mathcal{E}xt^1_{\pi_{S''}}\Big((u'\circ v',u\circ v)^*(\mathcal{E}_2,\mathcal{V}_2),(u'\circ v',u\circ v)^*
   (\mathcal{E}_1,\mathcal{V}_1)\Big)\Big). &
\end{eqnarray*}

Since we are assuming that $\mathcal{E}xt^1_{\pi_S}((\mathcal{E}_2,\mathcal{V}_2),(\mathcal{E}_1,\mathcal{V}_1))$
commutes with base change, so does $\mathcal{E}xt^1_{\pi_{S''}}((u'\circ v',u\circ v)^*(\mathcal{E}_2,\mathcal{V}_2),
(u'\circ v',u\circ v)^*(\mathcal{E}_1,\mathcal{V}_1))$. Therefore, $F$ is a contravariant functor from the category of
$S$-schemes to the category of sets.

\begin{proposition}\label{39}
Let us suppose that $\mathcal{E}xt^i_{\pi_S}((\mathcal{E}_2,\mathcal{V}_2),(\mathcal{E}_1,\mathcal{V}_1))$ commutes
with base change for $i=0,1$. Then the functor $F$ is represented by the vector bundle 

$$V:=\mathbb{V}\Big(\mathcal{E}xt^1_{\pi_S}\Big((\mathcal{E}_2,\mathcal{V}_2),(\mathcal{E}_1,\mathcal{V}_1)
\Big)^\vee\Big)\stackrel{\eta}{\longrightarrow} S$$
associated to the locally free sheaf $\mathcal{E}xt^1_{\pi_S}\left((\mathcal{E}_2,\mathcal{V}_2),(\mathcal{E}_1,
\mathcal{V}_1)\right)^\vee$. The fiber of $\eta$ over any point $s$ is canonically identified with $\operatorname{Ext}^1
((\mathcal{E}_2,\mathcal{V}_2)_s,(\mathcal{E}_1,\mathcal{V}_1)_s)$.
\end{proposition}

\begin{proof}
By hypothesis and base change for $i=1$, the sheaf $\hat{E}:=\mathcal{E}xt^1_{\pi_S}((\mathcal{E}_2,\mathcal{V}_2),
(\mathcal{E}_1,$ $\mathcal{V}_1))$ commutes with base change, so for every $S$-scheme $u:S'\rightarrow S$ we have that

$$F(S')=\operatorname{H}^0\Big(S',u^*\mathcal{E}xt^1_{\pi_S}\Big((\mathcal{E}_2,\mathcal{V}_2),(\mathcal{E}_1,
\mathcal{V}_1)\Big)\Big)=\operatorname{H}^0(S',u^*\hat{E}).$$

Moreover, using base change for $i=0,1$, we get that $\hat{E}$ is a locally free sheaf. Therefore, the functor $E$ is
represented by the vector bundle $V$ associated to $\hat{E}^\vee$ by the universal property of that object. Note that
by assumption $\hat{E}$ is locally free, so $\hat{E}^{\vee\vee}=\hat{E}$.
\end{proof}

\begin{remark}
The universal element of $F(V)$ is constructed in the following way. Let us consider the inclusion of sheaves on $S$
given by $\hat{E}^\vee\hookrightarrow\eta_*\mathcal{O}_{V}$ and the induced canonical inclusion

\begin{eqnarray*}
& \operatorname{H}^0(S,\textrm{End} \hat{E})=\operatorname{H}^0(S,\hat{E}\otimes \hat{E}^\vee)\hookrightarrow
  \operatorname{H}^0(S,\hat{E}\otimes\eta_*\mathcal{O}_{V})= &\\
& =\operatorname{H}^0(S,\eta_*\eta^*\hat{E})=\operatorname{H}^0(V,\eta^*\hat{E})=F(V). &
\end{eqnarray*}

Then we consider the image of the identity of $\hat{E}$ under this series of maps and we get that this is the universal
object for the functor $F$.
\end{remark}

By combining propositions \ref{35} and \ref{39} we get the following corollary.

\begin{corollary}\label{40}
Let us suppose that $S$ is \emph{reduced} and that $\mathcal{E}xt^i_{\pi_S}((\mathcal{E}_2,\mathcal{V}_2),$
$(\mathcal{E}_1,\mathcal{V}_1))$ commutes with base change for $i=0,1$. Then there is a family of extensions $\{e_v\}_{v\in V}$
of $(\eta',\eta)^*(\mathcal{E}_2,\mathcal{V}_2)$ by $(\eta',\eta)^*(\mathcal{E}_1,\mathcal{V}_1)$. Such a family is universal
over the category of \emph{reduced} $S$-schemes.
\end{corollary}

Here ``universal'' means the following: given any reduced $S$-scheme $u:S'\rightarrow S$ and any family of extensions
$\{e_{s'}\}_{s'\in S'}$ of $(u',u)^*(\mathcal{E}_2,\mathcal{V}_2)$ by $(u',u)^*(\mathcal{E}_1,\mathcal{V}_1)$
parametrized by $S'$, there is exactly one morphism $\psi: S'\rightarrow V$ of $S$-schemes such that $\{e_{s'}\}_{s'
\in S'}$ is the pullback of $\{e_v\}_{v\in V}$ via $\psi$. The relevant diagram to consider is given as follows:

\[\begin{tikzpicture}[xscale=1.5,yscale=-1.2]
    \node (A0_2) at (2, 0.6) {$\curvearrowright$};
    \node (A1_0) at (0, 1) {$C\times S'$};
    \node (A1_2) at (2, 1) {$C\times V$};
    \node (A1_4) at (4, 1) {$C\times S$};
    \node (A2_1) at (1, 2) {$\square$};
    \node (A2_3) at (3, 2) {$\square$};
    \node (A3_0) at (0, 3) {$S'$};
    \node (A3_2) at (2, 3) {$V$};
    \node (A3_4) at (4, 3) {$S.$};
    \node (A4_2) at (2, 3.4) {$\curvearrowright$};
    \path (A1_4) edge [->]node [auto] {$\scriptstyle{\pi_S}$} (A3_4);
    \path (A1_0) edge [->,bend right=35]node [auto] {$\scriptstyle{u'}$} (A1_4);
    \path (A3_0) edge [->]node [auto] {$\scriptstyle{\psi}$} (A3_2);
    \path (A1_0) edge [->]node [auto] {$\scriptstyle{\psi'}$} (A1_2);
    \path (A3_0) edge [->,bend left=35,swap]node [auto] {$\scriptstyle{u}$} (A3_4);
    \path (A1_0) edge [->]node [auto] {$\scriptstyle{\pi_{S'}}$} (A3_0);
    \path (A1_2) edge [->]node [auto] {$\scriptstyle{\eta'}$} (A1_4);
    \path (A1_2) edge [->]node [auto] {$\scriptstyle{\pi_V}$} (A3_2);
    \path (A3_2) edge [->]node [auto] {$\scriptstyle{\eta}$} (A3_4);
\end{tikzpicture}\]

\begin{corollary}\label{41}
Let us suppose that $\operatorname{Hom}((\mathcal{E}_2,\mathcal{V}_2)_s,(\mathcal{E}_1,\mathcal{V}_1)_s)=0$ for all
$s\in S$ and that $\mathcal{E}xt^1_{\pi_S}((\mathcal{E}_2,\mathcal{V}_2),(\mathcal{E}_1,\mathcal{V}_1))$ commutes with
base change. Then there is an extension parametrized by $V$

\begin{equation}\label{42}
0\rightarrow (\eta',\eta)^*(\mathcal{E}_1,\mathcal{V}_1)\rightarrow (\mathcal{E}_V,\mathcal{V}_V)\rightarrow
(\eta',\eta)^*(\mathcal{E}_2,\mathcal{V}_2)\rightarrow 0
\end{equation}
that is \emph{universal} on the category of $S$-schemes.
\end{corollary}

Here ``universal'' means the following: let us fix any $S$-scheme $u:S'\rightarrow S$ and any extension

\begin{equation}\label{43}
0\rightarrow (u',u)^*(\mathcal{E}_1,\mathcal{V}_1)\rightarrow (\mathcal{E}_{S'},\mathcal{V}_{S'})\rightarrow
(u',u)^*(\mathcal{E}_2,\mathcal{V}_2)\rightarrow 0
\end{equation}
over $S'$. Then there is a unique morphism $\psi: S'\rightarrow V$ of $S$-schemes such that (\ref{43}) is the pullback
of (\ref{42}) via $\psi$.

\begin{proof}
If we assume the hypotheses, then by lemma \ref{29} (a) for all morphisms $u:S'\rightarrow S$ we get a canonical
isomorphism

$$\mu:\operatorname{Ext}^1_{S'}\Big((u',u)^*(\mathcal{E}_2,\mathcal{V}_2),(u',u)^*(\mathcal{E}_1,\mathcal{V}_1)\Big)
\stackrel{\sim}{\rightarrow}$$
$$\stackrel{\sim}{\rightarrow}\operatorname{H}^0\Big(S',\mathcal{E}xt^1_{\pi_{S'}}\Big((u',u)^*(\mathcal{E}_2,
\mathcal{V}_2),(u',u)^*(\mathcal{E}_1,\mathcal{V}_1)\Big)\Big).$$

If we use proposition \ref{34} and the hypothesis, then this coincides also with

$$\operatorname{EXT}_{glob}\Big((u',u)^*(\mathcal{E}_2,\mathcal{V}_2),(u',u)^*(\mathcal{E}_1,\mathcal{V}_1)\Big).$$

So for every $S$-scheme $S'$ as before we can consider the set $F(S')$ as the set of all extensions of $(u',u)^*
(\mathcal{E}_2,\mathcal{V}_2)$ by $(u',u)^*(\mathcal{E}_1,\mathcal{V}_1)$ over $S'$. In particular, the universal
object of the functor $F$ corresponds to the extension (\ref{42}). Then the universal property of such an object
(together with the fact that $\mu$ is canonical) proves the claim.
\end{proof}

\section{Universal families of non-degenerate extensions}
Let us fix any integer $t\geq 1$ and let us suppose again that $\mathcal{E}xt^1_{\pi_S}((\mathcal{E}',\mathcal{V}'),$
$(\mathcal{E},\mathcal{V}))$ \emph{commutes with base change}. Then let us define a contravariant functor $G_t$ from
the category of $S$-schemes to the category of sets. For every morphism $u:S'\rightarrow S$, let us consider the
pullback diagram (\ref{8}) and let us define:

$$G_t(S'):=\Big\{\textrm{locally free quotients of rank }t\textrm{ of }$$
$$\mathcal{E}xt^1_{\pi_{S'}}\Big((u',u)^*(\mathcal{E}_2,
\mathcal{V}_2),(u',u)^*(\mathcal{E}_1,\mathcal{V}_1)\Big)^\vee\Big\}.$$

Let us fix any morphism $v: S''\rightarrow S'$ of $S$-schemes and any object of $F_t(S')$, i.e. any locally free
quotient of rank $t$:

$$\mathcal{E}xt^1_{\pi_{S'}}\Big((u',u)^*(\mathcal{E}_2,\mathcal{V}_2),(u',u)^*(\mathcal{E}_1,\mathcal{V}_1)\Big)^\vee
\longrightarrow\mathcal{M}\longrightarrow 0.$$

Then by pullback via $v$, we get an exact sequence:

\begin{equation}\label{48}
v^*\mathcal{E}xt^1_{\pi_{S'}}\Big((u',u)^*(\mathcal{E}_2,\mathcal{V}_2),(u',u)^*(\mathcal{E}_1,\mathcal{V}_1)
\Big)^\vee\longrightarrow v^*\mathcal{M}\longrightarrow 0.
\end{equation}

Using base change for $i=1$ we get:

\begin{eqnarray*}
& v^*\mathcal{E}xt^1_{\pi_{S'}}\Big((u',u)^*(\mathcal{E}_2,\mathcal{V}_2),(u',u)^*(\mathcal{E}_1,
   \mathcal{V}_1)\Big)^\vee= &\\
& =\Big(v^*\mathcal{E}xt^1_{\pi_{S'}}\Big((u',u)^*(\mathcal{E}_2,\mathcal{V}_2),(u',u)^*(\mathcal{E}_1,
   \mathcal{V}_1)\Big)\Big)^\vee\simeq &\\
& \simeq\mathcal{E}xt^1_{\pi_{S''}}\Big((u'\circ v',u\circ v)^*(\mathcal{E}_2,\mathcal{V}_2),
   (u'\circ v',u\circ v)^*(\mathcal{E}_1,\mathcal{V}_1)\Big)\Big)^\vee. &
\end{eqnarray*}

Therefore, (\ref{48}) gives an element of $G_t(S'')$, so we get a set map $G_t(v):G_t(S')\rightarrow G_t(S'')$.
Using base change for $i=1$, this gives rise to a contravariant functor $G_t$ on the category of $S$-schemes.

\begin{proposition}\label{49}
Let us suppose that $\mathcal{E}xt^i_{\pi_S}((\mathcal{E}_2,\mathcal{V}_2),(\mathcal{E}_1,\mathcal{V}_1))$
commutes with base change for $i=0,1$. Then for every $t\geq 1$ the functor $G_t$ is represented by the relative
Grassmannian of rank $t$

$$Q_t:=\operatorname{Grass}\Big(t,\mathcal{E}xt^1_{\pi_S}\Big((\mathcal{E}_2,\mathcal{V}_2),(\mathcal{E}_1,
\mathcal{V}_1)\Big)^\vee\Big)\stackrel{\theta_t}{\longrightarrow} S$$
associated to the locally free sheaf $\mathcal{E}xt^1_{\pi_S}((\mathcal{E}_2,\mathcal{V}_2),(\mathcal{E}_1,
\mathcal{V}_1))^\vee$ on $S$. The fiber of $\theta_t$ over any point $s$ is canonically identified with $\operatorname{Grass}
(t,\operatorname{Ext}^1((\mathcal{E}_2,\mathcal{V}_2)_s,(\mathcal{E}_1,\mathcal{V}_1)_s))$.
\end{proposition}

\begin{proof}
By hypothesis and base change for $i=0,1$, the sheaf $\hat{E}:=\mathcal{E}xt^1_{\pi_S}((\mathcal{E}_2,\mathcal{V}_2),$
$(\mathcal{E}_1,\mathcal{V}_1))$ commutes with base change and it is locally free. Therefore, for every
$S$-scheme $u:S'\rightarrow S$, $G_t(S')$ is equal to the set of locally free quotients of rank $t$ of
$u^*\hat{E}^\vee$. Now $G_t$ is represented by the Grassmannian bundle $\theta_t:\operatorname{Grass}(t,\hat{E}^\vee)
\rightarrow S$ by the universal property of the Grassmannian functor associated to every quasi-coherent
$\mathcal{O}_S$-module. Note that since $\hat{E}$ is locally free, then $\hat{E}^{\vee\vee}=\hat{E}$.
\end{proof}

\begin{remark}
In general we don't know how to explicitly describe the universal object of the functor $G_t$. We only know that it
will be something of the form

$$\mathcal{E}xt^1_{\pi_{Q_t}}\Big((\theta'_t,\theta_t)^*(\mathcal{E}_2,\mathcal{V}_2),(\theta'_t,\theta_t)^*
(\mathcal{E}_1,\mathcal{V}_1)\Big)^{\vee}\stackrel{q}{\longrightarrow}\overline{\mathcal{M}}_t\longrightarrow 0$$
for some locally free sheaf $\overline{\mathcal{M}}_t$ on $Q_t$ of rank $t$. It is reasonable that
$\overline{\mathcal{M}}_t$ is the very ample sheaf on $Q_t$ that induces the Pl\"{u}cker embedding of the relative
Grassmannian $Q_t$ into a projective space, but we don't have a proof of this fact (see the next section for the
special case $t=1$).
\end{remark}

\begin{defin}\label{62}
Let us fix any scheme $R$, any locally free $\mathcal{O}_R$-module $\mathcal{M}$ of rank $t$ and any exact sequence of
families of coherent systems of the form

\begin{equation}\label{50}
0\rightarrow(\mathcal{F}_1,\mathcal{Z}_1)\otimes_R\mathcal{M}\rightarrow(\mathcal{F},\mathcal{Z})\rightarrow
(\mathcal{F}_2,\mathcal{Z}_2)\rightarrow 0.
\end{equation}

By restriction to any fiber $C_r=C\times\{r\}$ over any point $r$ of $R$, we get a sequence that is a representative
for an object

\begin{eqnarray*}
& \xi_r\in\operatorname{Ext}^1\Big((\mathcal{F}_{2,r},\mathcal{Z}_{2,r}),(\mathcal{F}_{1,r},\mathcal{Z}_{1,r})\otimes_r
  \mathcal{M}_r\Big)= &\\
& =\operatorname{Ext}^1\Big((\mathcal{F}_{2,r},\mathcal{Z}_{2,r}),(\mathcal{F}_{1,r},\mathcal{Z}_{1,r})^{\oplus_t}\Big)
  =:H_r^{\oplus_t}. &
\end{eqnarray*}

So we can write $\xi_r=(\xi_r^1,\cdots,\xi_r^t)$. Then we say that (\ref{50}) is \emph{non-degenerate of rank $t$ on
the left} if for all points $r$ of $R$ the objects $\xi_r^i$ for $i=1,\cdots,t$ are linearly independent in $H_r$.
Analogously, we call non-degenerate on the left any family $\{e_r\}_{r\in R}$ of extensions of the same 2 objects on
the left and on the right of (\ref{50}) such that $e_r$ is non-degenerate for all $r\in R$. Similar definitions can
be given for \emph{non-degenerate (families of) extensions of rank $t$ on the right}.
\end{defin}

\begin{lemma}\label{51}
Let us assume the same hypotheses as for proposition \ref{49}. Then for every $S$-scheme $u:S'\rightarrow S$ we have
that $G_t(S')$ is the set of all the families of non-degenerate extensions of $(u',u)^*(\mathcal{E}_2,\mathcal{V}_2)$
by $(u',u)^*(\mathcal{E}_1,\mathcal{V}_1)\otimes_{S'}\mathcal{M}$ with arbitrary $\mathcal{M}$ locally free of rank
$t$ on $S'$, modulo the canonical operation of $\operatorname{H}^0(S',GL(t,$ $\mathcal{O}_{S'}))$.
\end{lemma}

\begin{proof}
By construction, $G_t(S')$ is equal to the set of all nowhere vanishing global sections of every sheaf on $S'$ of the
form

$$\mathcal{E}xt^1_{\pi_{S'}}\Big((u',u)^*(\mathcal{E}_2,\mathcal{V}_2),(u',u)^*(\mathcal{E}_1,\mathcal{V}_1)\Big)
\otimes_{S'}\mathcal{M}$$
with arbitrary $\mathcal{M}$ locally free of rank $t$ on $S'$, modulo the canonical operation of $\operatorname{H}^0
(S',GL(t,$ $\mathcal{O}_{S'}))$. Since every such $\mathcal{M}$ is locally free, we can use lemma \ref{28} and we
conclude by proposition \ref{35}.
\end{proof}

The proofs of the following two corollaries are modelled on the proofs of corollaries \ref{40} and \ref{41} together
with lemma \ref{51} and proposition \ref{49}, so we omit the details.

\begin{corollary}\label{52}
Let us fix any $t\geq 1$, let us suppose that $S$ is \emph{reduced} and that $\mathcal{E}xt^i_{\pi_S}((\mathcal{E}_2,
\mathcal{V}_2),(\mathcal{E}_1,\mathcal{V}_1))$ commutes with base change for $i=0,1$. Then there is a family
$\{e_q\}_{q\in Q_t}$ of non-degenerate extensions of rank $t$ on the left of $(\theta'_t,\theta_t)^*
(\mathcal{E}_2,\mathcal{V}_2)$ by $(\theta'_t,\theta_t)^*(\mathcal{E}_1,\mathcal{V}_1)\otimes_{Q_t}\overline{\mathcal{M}}_t$,
which is universal on the category of reduced $S$-schemes.
\end{corollary}

Here ``universal'' means the following: given any reduced $S$-scheme $u:S'\rightarrow S$, any locally free sheaf
$\mathcal{M}$ of rank $t$ on $S'$ and any class of a family $\{e_{s'}\}_{s'\in S'}$ of non-degenerate extensions of
rank $t$ on the left of $(u',u)^*(\mathcal{E}_2,\mathcal{V}_2)$ by $(u',u)^*(\mathcal{E}_1,\mathcal{V}_1)\otimes_{S'}
\mathcal{M}$, then there is a unique morphism of $S$-schemes $\psi:S'\rightarrow Q_t$, such that the family
$\{e_{s'}\}_{s'\in S'}$ is the pullback of $\{e_q\}_{q\in Q_t}$ via $\psi$, modulo the canonical operation of
$\operatorname{H}^0(S',GL(t,\mathcal{O}_{S'}))$. The relevant diagram to consider is

\[\begin{tikzpicture}[xscale=1.5,yscale=-1.2]
    \node (A0_2) at (2, 0.6) {$\curvearrowright$};
    \node (A1_0) at (0, 1) {$C\times S'$};
    \node (A1_2) at (2, 1) {$C\times Q_t$};
    \node (A1_4) at (4, 1) {$C\times S$};
    \node (A2_1) at (1, 2) {$\square$};
    \node (A2_3) at (3, 2) {$\square$};
    \node (A3_0) at (0, 3) {$S'$};
    \node (A3_2) at (2, 3) {$Q_t$};
    \node (A3_4) at (4, 3) {$S.$};
    \node (A4_2) at (2, 3.4) {$\curvearrowright$};
    \path (A1_4) edge [->]node [auto] {$\scriptstyle{\pi_S}$} (A3_4);
    \path (A1_0) edge [->,bend right=35]node [auto] {$\scriptstyle{u'}$} (A1_4);
    \path (A3_0) edge [->]node [auto] {$\scriptstyle{\psi}$} (A3_2);
    \path (A1_0) edge [->]node [auto] {$\scriptstyle{\psi'}$} (A1_2);
    \path (A3_0) edge [->,bend left=35,swap]node [auto] {$\scriptstyle{u}$} (A3_4);
    \path (A1_0) edge [->]node [auto] {$\scriptstyle{\pi_{S'}}$} (A3_0);
    \path (A1_2) edge [->]node [auto] {$\scriptstyle{\theta_t'}$} (A1_4);
    \path (A1_2) edge [->]node [auto] {$\scriptstyle{\pi_{Q_t}}$} (A3_2);
    \path (A3_2) edge [->]node [auto] {$\scriptstyle{\theta_t}$} (A3_4);
\end{tikzpicture}\]

\begin{corollary}\label{53}
Let us fix any $t\geq 1$, let us suppose that $\operatorname{Hom}((\mathcal{E}_2,\mathcal{V}_2)_s,
(\mathcal{E}_1,\mathcal{V}_1)_s)=0$ for all $s\in S$ and that $\mathcal{E}xt^1_{\pi_S}((\mathcal{E}_2,\mathcal{V}_2),
(\mathcal{E}_1,\mathcal{V}_1))$ commutes with base change. Then there is a family $(\mathcal{E}_{Q_t},
\mathcal{V}_{Q_t})$ parametrized by $Q_t$ and a non-degenerate extension of rank $t$ on the left

\begin{equation}\label{4-46}
0\rightarrow(\theta'_t,\theta_t)^*(\mathcal{E}_1,\mathcal{V}_1)\otimes_{Q_t}\overline{\mathcal{M}}_t\rightarrow
(\mathcal{E}_{Q_t},\mathcal{V}_{Q_t})\rightarrow(\theta'_t,\theta_t)^*(\mathcal{E}_2,\mathcal{V}_2)\rightarrow 0
\end{equation}
that is \emph{universal} on the category of $S$-schemes.
\end{corollary}

Here ``universal'' means the following: let us fix any morphism $u:S'\rightarrow S$, any locally free sheaf
$\mathcal{M}$ of rank $t$ on $S'$ and any non-degenerate extension of rank $t$ on the left:

\begin{equation}\label{4-47}
0\rightarrow (u',u)^*(\mathcal{E}_1,\mathcal{V}_1)\otimes_{S'}\mathcal{M}\rightarrow(\mathcal{E}_S,\mathcal{V}_S)
\rightarrow(u',u)^*(\mathcal{E}_2,\mathcal{V}_2)\rightarrow 0.
\end{equation}

Then there is a unique morphism of $S$-schemes $\psi:S'\rightarrow Q_t$ such that (\ref{4-47}) is the pullback of
(\ref{4-46}) via $\psi$, modulo the canonical operation of $\operatorname{H}^0(S',GL(t,\mathcal{O}_{S'}))$.\\

Analogously, using the second part of lemma \ref{28} we can prove the following results.

\begin{corollary}\label{4-48}
Let us fix any $t\geq 1$, let us suppose that $S$ is \emph{reduced} and that $\mathcal{E}xt^i_{\pi_S}
((\mathcal{E}_2,\mathcal{V}_2),$ $(\mathcal{E}_1,\mathcal{V}_1))$ commutes with base change for $i=0,1$. Then there is
a family of non-degenerate extensions of rank $t$ on the right $\{e'_q\}_{q\in Q_t}$ of $(\theta'_t,\theta_t)^*
(\mathcal{E}_2,\mathcal{V}_2)$ $\otimes_{Q_t}\overline{\mathcal{M}}_t^{\vee}$ by $(\theta'_t,\theta_t)^*
(\mathcal{E}_1,\mathcal{V}_1)$, which is universal for families of non-degenerate extensions of rank $t$ on the right,
analogously to corollary \ref{52}.
\end{corollary}

\begin{corollary}\label{4-49}
Let us fix any $t\geq 1$, let us suppose that $\operatorname{Hom}((\mathcal{E}_2,\mathcal{V}_2)_s,
(\mathcal{E}_1,\mathcal{V}_1)_s)=0$ for all $s\in S$ and that $\mathcal{E}xt^1_{\pi_S}((\mathcal{E}_2,\mathcal{V}_2),
(\mathcal{E}_1,\mathcal{V}_1))$commutes with base change. Then there is a family $(\mathcal{E}'_{Q_t},
\mathcal{V}'_{Q_t})$ parametrized by $Q_t$ and a non-degenerate extension on the right of rank $t$

\begin{equation}\label{4-50}
0\rightarrow(\theta'_t,\theta_t)^*(\mathcal{E}_1,\mathcal{V}_1)\rightarrow (\mathcal{E}'_{Q_t},\mathcal{V}'_{Q_t})
\rightarrow(\theta'_t,\theta_t)^*(\mathcal{E}_2,\mathcal{V}_2)\otimes_{Q_t}\overline{\mathcal{M}}^{\vee}_t\rightarrow 0
\end{equation}
that is universal for non-degenerate extensions of rank $t$ on the right, analogously to corollary \ref{53}.
\end{corollary}

\section{Universal families of non-split extensions}
If we fix $t=1$ in the previous section and we simplify the notations by setting $P:=Q_1$ and $\varphi:=\theta_1$, we
get:

\begin{corollary}
Let us suppose that $\mathcal{E}xt^i_{\pi_S}((\mathcal{E}_2,\mathcal{V}_2),(\mathcal{E}_1,\mathcal{V}_1))$ commutes
with base change for $i=0,1$. Then the functor $F_1$ is represented by the projective bundle 

$$P:=\mathbb{P}\Big(\mathcal{E}xt^1_{\pi_S}\Big((\mathcal{E}_2,\mathcal{V}_2),(\mathcal{E}_1,\mathcal{V}_1)\Big)^\vee
\Big)\stackrel{\varphi}{\longrightarrow} S$$
associated to the locally free sheaf $\mathcal{E}xt^1_{\pi_S}((\mathcal{E}_2,\mathcal{V}_2),
(\mathcal{E}_1,\mathcal{V}_1))^\vee$ on $S$. The fiber of $\varphi$ over any point $s$ is canonically identified with
$\mathbb{P}(\operatorname{Ext}^1((\mathcal{E}_2,\mathcal{V}_2)_s,(\mathcal{E}_1,\mathcal{V}_1)_s))$.
\end{corollary}

\begin{remark}
The universal element of $F_1(P)$ is constructed in the following way. We consider the canonical isomorphisms:

\begin{eqnarray*}
& \operatorname{H}^0(S,\textrm{End} \hat{E})=\operatorname{H}^0(S,\hat{E}\otimes \hat{E}^\vee)=\operatorname{H}^0(S,
  \hat{E}\otimes\varphi_*\mathcal{O}_P(1))= &\\
& =\operatorname{H}^0\Big(S,\varphi_*\Big(\varphi^*\hat{E}\otimes\mathcal{O}_P(1)\Big)\Big)=
  \operatorname{H}^0(P,\varphi^*\hat{E}\otimes_P\mathcal{O}_P(1)). &
\end{eqnarray*}

Then we consider the image of the identity of $\hat{E}$ under this series of isomorphisms and we get that this is
a non-vanishing section of $\varphi^*\hat{E}\otimes_P\mathcal{O}_P(1)$. Using base change for $i=1$, this gives a
non-vanishing section of

$$\mathcal{E}xt^1_{\pi_P}\Big((\varphi',\varphi)^*(\mathcal{E}_2,\mathcal{V}_2),
(\varphi',\varphi)^*(\mathcal{E}_1,\mathcal{V}_1)\Big)\otimes_P\mathcal{O}_P(1),$$
so it defines a quotient:

$$\mathcal{E}xt^1_{\pi_P}\Big((\varphi',\varphi)^*(\mathcal{E}_2,\mathcal{V}_2),(\varphi,\varphi)^*(\mathcal{E}_1,
\mathcal{V}_1)\Big)^\vee\longrightarrow\mathcal{O}_P(1)\longrightarrow 0.$$

This is the universal object of the functor $F_1$.
\end{remark}

The notion of (family of) non-degenerate extension(s) (either on the left or on the right) of rank $t=1$ coincides
with the notion of (family of) non-split extension(s). Therefore we get the following corollaries.

\begin{corollary}\label{44}
Let us suppose that $S$ is \emph{reduced} and that $\mathcal{E}xt^i_{\pi_S}((\mathcal{E}_2,\mathcal{V}_2),$
$(\mathcal{E}_1,\mathcal{V}_1))$ commutes with base change for $i=0,1$. Then there is a family of non-split extensions
$\{e_p\}_{p\in P}$ of $(\varphi',\varphi)^*(\mathcal{E}_2,\mathcal{V}_2)$ by $(\varphi',\varphi)^*(\mathcal{E}_1,
\mathcal{V}_1)\otimes_P\mathcal{O}_P(1)$, which is universal on the category of reduced $S$-schemes.
\end{corollary}

Here ``universal'' means the following: given any reduced $S$-scheme $u:S'\rightarrow S$, any $\mathcal{L}\in
\textrm{Pic}(S')$ and any family $\{e_{s'}\}_{s'\in S'}$ of non-split extensions of $(u',u)^*(\mathcal{E}_2,
\mathcal{V}_2)$ by $(u',u)^*(\mathcal{E}_1,\mathcal{V}_1)\otimes_{S'} \mathcal{L}$ over $S'$, then there is a unique
morphism of $S$-schemes $\psi:S'\rightarrow P$ such that the family $\{e_{s'}\}_{s'\in S'}$ is the pullback of
$\{e_p\}_{p\in P}$ via $\psi$, modulo the canonical operation of $\operatorname{H}^0(S',\mathcal{O}^*_{S'})$. The
relevant diagram to consider is the following:

\[\begin{tikzpicture}[xscale=1.5,yscale=-1.2]
    \node (A0_2) at (2, 0.6) {$\curvearrowright$};
    \node (A1_0) at (0, 1) {$C\times S'$};
    \node (A1_2) at (2, 1) {$C\times P$};
    \node (A1_4) at (4, 1) {$C\times S$};
    \node (A2_1) at (1, 2) {$\square$};
    \node (A2_3) at (3, 2) {$\square$};
    \node (A3_0) at (0, 3) {$S'$};
    \node (A3_2) at (2, 3) {$P$};
    \node (A3_4) at (4, 3) {$S.$};
    \node (A4_2) at (2, 3.4) {$\curvearrowright$};
    \path (A1_4) edge [->]node [auto] {$\scriptstyle{\pi_S}$} (A3_4);
    \path (A1_0) edge [->,bend right=35]node [auto] {$\scriptstyle{u'}$} (A1_4);
    \path (A3_0) edge [->]node [auto] {$\scriptstyle{\psi}$} (A3_2);
    \path (A1_0) edge [->]node [auto] {$\scriptstyle{\psi'}$} (A1_2);
    \path (A3_0) edge [->,bend left=35,swap]node [auto] {$\scriptstyle{u}$} (A3_4);
    \path (A1_0) edge [->]node [auto] {$\scriptstyle{\pi_{S'}}$} (A3_0);
    \path (A1_2) edge [->]node [auto] {$\scriptstyle{\varphi'}$} (A1_4);
    \path (A1_2) edge [->]node [auto] {$\scriptstyle{\pi_P}$} (A3_2);
    \path (A3_2) edge [->]node [auto] {$\scriptstyle{\varphi}$} (A3_4);
\end{tikzpicture}\]

\begin{corollary}\label{45}
Let us suppose that $\operatorname{Hom}((\mathcal{E}_2,\mathcal{V}_2)_s,(\mathcal{E}_1,\mathcal{V}_1)_s)=0$ for all
$s\in S$ and that $\mathcal{E}xt^1_{\pi_S}((\mathcal{E}_2,\mathcal{V}_2),(\mathcal{E}_1,\mathcal{V}_1))$ commutes with
base change. Then there is a family $(\mathcal{E}_P,\mathcal{V}_P)$ parametrized by $P$ and a non-split extension

\begin{equation}\label{46}
0\rightarrow(\varphi',\varphi)^*(\mathcal{E}_1,\mathcal{V}_1)\otimes_P\mathcal{O}_P(1)\rightarrow
(\mathcal{E}_P,\mathcal{V}_P)\rightarrow(\varphi',\varphi)^*(\mathcal{E}_2,\mathcal{V}_2)\rightarrow 0
\end{equation}
that is universal on the category of $S$-schemes.
\end{corollary}

Here ``universal'' means the following: let us fix any morphism $u:S'\rightarrow S$, any line bundle $\mathcal{L}\in
\textrm{Pic}(S')$ and any non-split extension

\begin{equation}\label{47}
0\rightarrow (u',u)^*(\mathcal{E}_1,\mathcal{V}_1)\otimes_{S'}\mathcal{L}\rightarrow(\mathcal{E}_S,\mathcal{V}_S)
\rightarrow(u',u)^*(\mathcal{E}_2,\mathcal{V}_2)\rightarrow 0.
\end{equation}

Then there is a unique morphism of $S$-schemes $\psi:S'\rightarrow P$ such that (\ref{47}) is the pullback of
(\ref{46}) via $\psi$, modulo the canonical operation of $\operatorname{H}^0(S',\mathcal{O}_{S'}^*)$.

\section{Applications}

Let us fix any scheme $T$ and any pair of families of coherent systems $(\mathcal{E}_l,\mathcal{V}_l)$ parametrized by
$T$ (of type $(n_l,d_l,k_l)$) for $l=1,2$.

\begin{defin}
For any pair $(a,b)\in\mathbb{N}_0$ we set:

$$T_{a,b}:=\Big\{t\in T\textrm{ s.t. dim }\operatorname{Ext}^2((\mathcal{E}_2,\mathcal{V}_2)_t,(\mathcal{E}_1,
\mathcal{V}_1)_t)=:a,$$
$$\textrm{dim }\operatorname{Hom}((\mathcal{E}_2,\mathcal{V}_2)_t,(\mathcal{E}_1,\mathcal{V}_1)_t)=:b\Big\}.$$
\end{defin}

By proposition \ref{11}, each set $T_{a,b}$ is locally closed in $T$ with the induced reduced structure. By
proposition \ref{16} we have that on $T_{a,b}$

\begin{equation}\label{54}
\textrm{dim Ext}^1((\mathcal{E}_2,\mathcal{V}_2)_t,(\mathcal{E}_1,\mathcal{V}_1)_t)=:C_{21}+a+b,
\end{equation}
so it is constant (here $C_{21}$ depends only on the genus of $C$ and on $(n_l,d_l,k_l)$ for $l=1,2$). By
\cite[lemma 3.3]{BGMN} the quantity (\ref{54}) is bounded above on $T$ (and it is non-negative). Since both $a$ and
$b$ are non-negative integers, then $T$ is the disjoint union of finitely many non-empty locally closed subschemes of
the form $T_{a,b}$.

\begin{lemma}\label{55}
Each $T_{a,b}$ admits a finite stratification $\{T_{a,b}^j\}_j$ consisting of locally closed reduced subschemes and
such that on each $T_{a,b}^j$ the sheaves

$$\mathcal{E}xt^i_{\pi_{T_{a,b}^j}} \Big((\mathcal{E}_2,\mathcal{V}_2)|_{T_{a,b}^j},
(\mathcal{E}_1,\mathcal{V}_1)|_{T_{a,b}^j}\Big)\quad\textrm{for }i=0,1,2$$
are locally free and commute with every base change to $T_{a,b}^j$. If $T_{a,b}$ is integral for a certain $(a,b)$,
then the set $\{T_{a,b}^j\}_j$ can be chosen to coincide with $\{T_{a,b}\}$. 
\end{lemma}

\begin{proof}
Let us fix any pair $(a,b)$ such that $T_{a,b}\neq\varnothing$ and let us consider the set $\{T_{a,b;l}\}_l$ of its
irreducible components; since we are working with schemes of finite type over $\mathbb{C}$, such a set is finite. By
construction every $T_{a,b;l}$ is reduced and irreducible, hence integral. Now for every triple $(a,b;l)$, for every
$i\geq 0$ and for every $t\in T_{a,b;l}$, let us denote by $\tau^i(a,b;l;t)$ the base change:

$$\tau^i(a,b;l;t):\mathcal{E}xt^i_{\pi_{T_{a,b;l}}}\Big((\mathcal{E}_2,\mathcal{V}_2)|_{T_{a,b;l}},
(\mathcal{E}_1,\mathcal{V}_1)|_{T_{a,b;l}}\Big)\otimes k(t)\rightarrow$$
$$\rightarrow\operatorname{Ext}^i\Big((\mathcal{E}_2,\mathcal{V}_2)_t,(\mathcal{E}_1,\mathcal{V}_1)_t\Big).$$

Since $C$ is a curve, for every point $t$ in $T$ we have that

$$\operatorname{Ext}^3\Big((\mathcal{E}_2,\mathcal{V}_2)_t,(\mathcal{E}_1,\mathcal{V}_1)_t\Big)=0.$$

Therefore, every $\tau^3(a,b,l,t)$ is surjective and $\mathcal{E}xt^3_{\pi_{T_{a,b;l}}}((\mathcal{E}_2,
\mathcal{V}_2)|_{T_{a,b;l}},(\mathcal{E}_1,\mathcal{V}_1)|_{T_{a,b;l}})$ $=0$, so in particular it is locally free. Now
by (\ref{54}) we get that  for every $i=0,1,2$ the dimension of $\operatorname{Ext}^i((\mathcal{E}_2,\mathcal{V}_2)_t,
(\mathcal{E}_1,\mathcal{V}_1)_t)$ is constant on every $T_{a,b;l}$. Since every $T_{a,b;l}$ is integral, then by
proposition \ref{11} we get that on each $T_{a,b;l}$ the sheaves

$$\mathcal{E}xt^i_{\pi_{T_{a,b;l}}}\Big((\mathcal{E}_2,\mathcal{V}_2)|_{T_{a,b;l}},
(\mathcal{E}_1,\mathcal{V}_1)|_{T_{a,b;l}}\Big)$$
are locally free for $i=0,1,2$. Then by descending induction and base change (proposition \ref{27}) we can prove that
for every $i=0,1,2$, for every triple $(a,b;l)$ and for every $t$ in $T_{a,b;l}$ the base change $\tau^i(a,b;l;t)$
is an isomorphism.\\

Now let us fix any pair $(a,b)$ and let us denote by $L=\{l_1<\cdots<l_r\}$ the corresponding set of indices. For
each subset $\{l'_1< \cdots < l'_s\}\subset L$ we denote by $\{l'_{s+1} <\cdots <l'_{r}\}$ its complement in $L$ and
we define

\begin{equation}\label{56}
T_{a,b}^{l'_1,\cdots,l'_s}:=(T_{a,b;l'_1}\cap\cdots\cap T_{a,b;l'_s})\smallsetminus 
(T_{a,b;l'_{s+1}}\cup\cdots\cup T_{a,b;l'_r}).
\end{equation}

Each such scheme is locally closed in $T$ and any two such schemes are disjoint if they are associated to different
sets of indices; moreover each $T_{a,b}$ is covered by such subschemes. Then we denote by $j$ any set of indices
$j:=\{l'_1< \cdots < l'_s\}$ and by $T_{a,b}^j$ the corresponding scheme defined as in (\ref{56}). For each $(a,b)$,
the set of all such $j$'s is finite. Now for each such $j$, let us consider the inclusion $T_{a,b}^j\hookrightarrow
T_{a,b;l'_1}$. By base change for $i=0,1,2$ the sheaves 

$$\mathcal{E}xt^i_{\pi_{T_{a,b}^j}}\Big((\mathcal{E}_2,\mathcal{V}_2)|_{T_{a,b}^j},
(\mathcal{E}_1,\mathcal{V}_1)|_{T_{a,b}^j}\Big)=$$
$$=\left(\mathcal{E}xt^i_{\pi_{T_{a,b;l'_1}}}\Big((\mathcal{E}_2,\mathcal{V}_2)|_{T_{a,b;l'_1}},
(\mathcal{E}_1,\mathcal{V}_1)|_{T_{a,b}^{l'_1}}\Big)\right)|_{T_{a,b}^j}$$
are locally free for $i=0,1,2$ and commute with base change, so we conclude.
\end{proof}

By using lemma \ref{55} together with the results of the previous sections we get the following propositions.

\begin{proposition}\label{57}
Let us fix any scheme $T$ and any pair of families of coherent systems $(\mathcal{E}_l,\mathcal{V}_l)$ parametrized by
$T$ for $l=1,2$. Then there exists a finite stratification of $T$ by reduced locally closed subschemes $T_{a,b}^j$
defined as in lemma \ref{55}, such that the conclusions of corollaries \ref{40}, \ref{52}, \ref{4-48} and \ref{44}
hold for each $S=T_{a,b}^j$ and for the pair of families $(\mathcal{E}_l,\mathcal{V}_l)|_{T_{a,b}^j}$ for $l=1,2$. Let
us denote by

$$\eta_{a,b}^j:V_{a,b}^j\rightarrow T_{a,b}^j,\quad\theta_{t;a,b}^j:Q_{t;a,b}^j\rightarrow T_{a,b}^j,\quad
\varphi_{a,b}^j:P_{a,b}^j\rightarrow T_{a,b}^j$$
the vector bundles, Grassmannian fibrations (for $t\geq 2$) and the projective bundles obtained by those corollaries.
For every point $t$ of $T$ we write

$$\mathbb{H}^1_{21}(t):=\operatorname{Ext}^1((\mathcal{E}_2,\mathcal{V}_2)_t,(\mathcal{E}_1,\mathcal{V}_1)_t).$$

Then the dimension of the vector space $\mathbb{H}_{21}(t)$ is constant over each $T_{a,b}^j$. Moreover, the fibers
of $\eta_{a,b}^j$, $\theta_{t;a,b}^j$ and $\varphi_{a,b}^j$ over any $t\in T_{a,b}^j$ are canonically identified with
$\mathbb{H}^1_{21}(t)$, $\operatorname{Grass}(t,\mathbb{H}^1_{21}(t))$ and $\mathbb{P}(\mathbb{H}^1_{21}(t))$
respectively. T
\end{proposition}

\begin{proposition}
Let us fix any scheme $T$ and any pair of families of coherent systems $(\mathcal{E}_l,\mathcal{V}_l)$ parametrized by
$T$ for $l=1,2$. Let us suppose that $\operatorname{Hom}((\mathcal{E}_2,\mathcal{V}_2)_t,(\mathcal{E}_1,$
$\mathcal{V}_1)_t)=0$ for all $t\in T$. Then there exists a finite stratification of $T$ by reduced locally closed
subschemes $T_a^j$, such that the conclusions of corollaries \ref{41}, \ref{53}, \ref{4-49} and \ref{45} hold for each
$S=T_a^j$ and for the pair of families $(\mathcal{E}_l,\mathcal{V}_l)|_{T_a^j}$ for $l=1,2$. The description of the
various fibrations that are obtained in this way is the same as in the previous proposition, once we set $b:=0$ and
$T_{a,0}^j=:T_a^j$.
\end{proposition}

As we said in the introduction, the main motivation for studying universal families of extensions is that of giving
a scheme theoretic description of the sets of the form $G(\alpha_c;n,d,k;n_1,k_1)$ introduced in definition \ref{14}.\\

Let us fix any triple $(n,d,k)$, any critical value $\alpha_c$ for it and any pair $(n_1,k_1)$ as in definition
\ref{14}. Then let us set $d_1,n_2,d_2$ and $k_2$ as in that definition; moreover let us set $G_l:=
G(\alpha_c;n_l,d_l,k_l)$ for $l=1,2$. For $l=1,2$, let us denote by $\hat{G}_l$ the Quot schemes whose GIT quotient by
$PGL(N_l)$ is $G_l$ and let $(\hat{\mathcal{Q}}_l,\hat{\mathcal{W}}_l)$ be the local universal family parametrized
by $\hat{G}_l$ (see remark \ref{10}). If $GCD(n_l,d_l,k_l)=1$, let us denote by $(\mathcal{Q}_l,\mathcal{W}_l)$ the
corresponding universal family parametrized by $G_l$. In addition, let us denote by $\hat{p}_l:\hat{G}_1\times
\hat{G}_2\rightarrow\hat{G}_l$ and $p_l:G_1\times G_2\rightarrow G_1$ the various projections. We denote by $\hat{t}_l$
any point of $\hat{G}_l$ and by $t_l=(E_l,V_l)$ its image in $G_l$.

\begin{proposition}\label{58}
Having fixed all these notations, for all $(\alpha_c;n,d,k;n_1,k_1)$ as before there exists a finite stratification
$\{\hat{T}_{a,b;i}\}_{a,b;i}$ of $\hat{G}_1\times\hat{G}_2$ by locally closed subschemes such that:

\begin{itemize}
  \item $(a,b)$ varies over a finite subset of $\mathbb{N}_0^2$; for each $(a,b)$ the set $\{\hat{T}_{a,b;i}\}_i$ is a
  finite stratification by locally closed subschemes of
  
  $$\hat{T}_{a,b}:=\Big\{(\hat{t}_1,\hat{t}_2)\in\hat{G}_1\times\hat{G}_2\textrm{ s.t. }\operatorname{dim }
  \operatorname{Ext}^2((E_2,V_2),(E_1,V_1))=a,$$
  \begin{equation}\label{59}
  \operatorname{dim }\operatorname{Hom}((E_2,V_2),(E_1,V_1))=b\Big\}.
  \end{equation}
  
  Every $\hat{T}_{a,b;i}$ is invariant under the action of $PGL(N_1)\times PGL(N_2)$; if we denote
  by $T_{a,b;i}$ its image in $G_1\times G_2$, then $\{T_{a,b;i}\}_i$ is a finite stratification by locally closed
  subschemes of 
  
   $$T_{a,b}:=\Big\{((E_1,V_1),(E_2,V_2))\in G_1\times G_2\textrm{ s.t. }\operatorname{dim }\operatorname{Ext}^2(
   (E_2,V_2),(E_1,V_1))=a,$$
  \begin{equation}\label{60}
  \operatorname{dim }\operatorname{Hom}((E_2,V_2),(E_1,V_1))=b\Big\}.
  \end{equation}
  
 \item For each $(a,b;i)$ there exists a projective bundle $\hat{\varphi}_{a,b;i}:\hat{P}_{a,b;i}\rightarrow
   \hat{T}_{a,b;i}$, where
 
   $$\hat{P}_{a,b;i}:=\mathbb{P}\Big(\mathcal{E}xt^1_{\pi_{\hat{T}_{a,b;i}}}
   \Big((\hat{p}'_2,\hat{p}_2)^*(\hat{\mathcal{Q}}_2,\hat{\mathcal{W}}_2)|_{\hat{T}_{a,b;i}},
   (\hat{p}'_1,\hat{p}_1)^*(\hat{\mathcal{Q}}_1,\hat{\mathcal{W}}_1)|_{\hat{T}_{a,b;i}}\Big)\Big).$$
 
 \item There are free actions of $PGL(N_1)\times PGL(N_2)$ on the source and target of each $\hat{\varphi}_{a,b;i}$;
  there exists quotient schemes $G(\alpha_c;n,d,k;n_1,k_1;a,b;i)$ and $T_{a,b;i}$ and an induced fibration:
  
  $$\varphi_{a,b;i}:G(\alpha_c;n,d,k;n_1,k_1;a,b;i)\rightarrow T_{a,b;i}.$$
 
 \item For every point $t=((E_1,V_1),(E_2,V_2))\in T_{a,b;i}$ the fiber of $\varphi_{a,b;i}$ over it is given by
  $\mathbb{P}(\operatorname{Ext}^1((E_2,V_2),(E_1,V_1)))$;

 \item The set $G(\alpha_c;n,d,k)$ admits a finite stratification

  $$G(\alpha_c;n,d,k;n_1,k_1)=\coprod_{a,b;i} G(\alpha_c;n,d,k;n_1,k_1;a,b;i),$$
 
 \item For every $(a,b;i)$ there exists a universal extension parametrized by $\hat{P}_{a,b;i}$:
 
  $$0\rightarrow(\hat{\varphi}'_{a,b;i},\hat{\varphi}_{a,b;i})^*(\hat{p}'_1,\hat{p}_1)^*(\hat{\mathcal{Q}}_1,
  \hat{\mathcal{W}}_1)\otimes_{\hat{P}_{a,b;i}}\mathcal{O}_{\hat{P}_{a,b;i}}(1)\rightarrow$$
  \begin{equation}\label{61}
  \rightarrow(\hat{\mathcal{E}}_{a,b;i},\hat{\mathcal{V}}_{a,b;i})\rightarrow(\hat{\varphi}'_{a,b;i},
  \hat{\varphi}_{a,b;i})^*(\hat{p}'_2,\hat{p}_2)^*(\hat{\mathcal{Q}}_2,\hat{\mathcal{W}}_2)\rightarrow 0.
  \end{equation}
\end{itemize}

In addition, if $GCD(n_l,d_l,k_l)=1$ for $l=1,2$, then we can write

$$G(\alpha_c;n,d,k;n_1,k_1;a,b;i)=$$
$$=\mathbb{P}\Big(\mathcal{E}xt^1_{\pi_{T_{a,b;i}}}\Big((p'_2,p_2)^*
(\mathcal{Q}_2,\mathcal{W}_2)|_{T_{a,b;i}},(p'_1,p_1)^*(\mathcal{Q}_1,\mathcal{W}_1)|_{T_{a,b;i}}\Big)\Big)$$

and there exists a universal extension parametrized by $G(\alpha_c;n,d,k;n_1,k_1;a,b;i)$:

$$0\rightarrow(\varphi'_{a,b;i},\varphi_{a,b;i})^*(p'_1,p_1)^*(\mathcal{Q}_1,\mathcal{W}_1)\otimes_{P_{a,b;i}}
\mathcal{O}_{P_{a,b;i}}(1)\rightarrow$$
\begin{equation}\label{71}
\rightarrow(\mathcal{E}_{a,b;i},\mathcal{V}_{a,b;i})\rightarrow(\varphi'_{a,b;i},\varphi_{a,b;i})^*
(p'_2,p_2)^*(\mathcal{Q}_2,\mathcal{W}_2)\rightarrow 0.
\end{equation}

\end{proposition}

\begin{proof}
Let us set $\hat{T}:=\hat{G}_1\times\hat{G}_2$ and let us consider the families $(\mathcal{E}_l,\mathcal{V}_l):=
(\hat{p}'_l,\hat{p}_l)^*(\hat{\mathcal{Q}}_l,$ $\hat{\mathcal{W}}_l)$ parametrized by $T$ for $l=1,2$. By applying
proposition \ref{57} we get the following facts:

\begin{itemize}
 \item $\hat{T}$ has a finite stratification $\{\hat{T}_{a,b}^j\}_{a,b;j}$ by locally closed subschemes; by the local
  universal property of the families $(\hat{\mathcal{Q}}_l,\hat{\mathcal{W}}_l)$ for $l=1,2$ each set
  $\{\hat{T}_{a,b}^j\}_j$ is a finite stratification of the set $\hat{T}_{a,b}$ described in (\ref{59});
 \item for each $(a,b;j)$ there exists a projective bundle $\hat{\varphi}_{a,b}^j:\hat{P}_{a,b}^j\rightarrow
  \hat{T}_{a,b}^j$, where
 
  $$\hat{P}_{a,b}^j:=\mathbb{P}\Big((\hat{\mathcal{H}}_{a,b}^j)^{\vee}\Big)$$
 
  and the sheaf
 
  $$\hat{\mathcal{H}}_{a,b}^j:=\mathcal{E}xt^1_{\pi_{\hat{T}_{a,b}^j}}\Big((\hat{p}'_2,\hat{p}_2)^*(
  \hat{\mathcal{Q}}_2,\hat{\mathcal{W}}_2)|_{\hat{T}_{a,b}^j},(\hat{p}'_1,\hat{p}_1)^*(\hat{\mathcal{Q}}_1,
  \hat{\mathcal{W}}_1)|_{\hat{T}_{a,b}^j}\Big)$$
 
  commutes with base change. In particular, for every point $\hat{t}=(\hat{t}_1,\hat{t}_2)$ in $\hat{T}_{a,b}^j$ with
  image $((E_1,V_1),(E_2,V_2))$ in $G_1\times G_2$, the fiber of $\hat{\varphi}_{a,b}^j$ over $\hat{t}$ is canonically 
  identified with $\mathbb{P}(\operatorname{Ext}^1((E_2,V_2),(E_1,V_1)))$;
 
 \item for each $(a,b;j)$ there exists a family of non-split extensions $\{e_{\hat{p}}\}_{\hat{p}\in\hat{P}_{a,b}^j}$
  of $(\hat{\varphi}_{a,b}^{'j},\hat{\varphi}_{a,b}^j)^*(\hat{p}'_2,\hat{p}_2)^*(\hat{\mathcal{Q}}_2,
  \hat{\mathcal{W}}_2)$ by $((\hat{\varphi}_{a,b}^{'j},\hat{\varphi}_{a,b}^j)^*(\hat{p}'_1,\hat{p}_1)^*
  (\hat{\mathcal{Q}}_1,\hat{\mathcal{W}}_1))\otimes_{\hat{P}_{a,b}^j}\mathcal{O}_{\hat{P}_{a,b}^j}(1)$, which is
  universal on the category of reduced $\hat{T}_{a,b}^j$-schemes.
\end{itemize}

In particular, by definition of family there is an open covering $\{\hat{T}_{a,b}^{j,m}\}_m$ of $\hat{T}_{a,b}^j$ such
that for each $m$ there is an extension

$$0\rightarrow((\hat{\varphi}_{a,b}^{'j},\hat{\varphi}_{a,b}^j)^*(\hat{p}'_1,\hat{p}_1)^*(\hat{\mathcal{Q}}_1,
\hat{\mathcal{W}}_1))\otimes_{\hat{P}_{a,b}^j}\mathcal{O}_{\hat{P}_{a,b}^j}(1)|_{\hat{T}_{a,b}^{j,m}}\rightarrow$$
\begin{equation}\label{81}
\rightarrow(\hat{\mathcal{E}}_{a,b}^{j,m},\hat{\mathcal{V}}_{a,b}^{j,m})\rightarrow
(\hat{\varphi}_{a,b}^{'j},\hat{\varphi}_{a,b}^j)^*(\hat{p}'_2,\hat{p}_2)^*(\hat{\mathcal{Q}}_2,
\hat{\mathcal{W}}_2)|_{\hat{T}_{a,b}^{j,m}}\rightarrow 0
\end{equation}
that is universal on the category of reduced $\hat{T}_{a,b}^{j,m}$-schemes and that for every point of
$\hat{T}_{a,b}^{j,m}$ restricts to a non-split extension. Since we are working with schemes of finite type over
$\mathbb{C}$, then we can assume that $m$ varies over a finite set $M=\{m_1<\cdots<m_r\}$. Then for every subset
$m_{\bullet}=\{m'_1<\cdots<m'_s\}\subset M$ with complement set $\{m'_{s+1}<\cdots<m'_r\}$ we set

$$\hat{T}_{a,b}^{j,m_{\bullet}}:=(\hat{T}_{a,b}^{j,m'_1}\cap\cdots\cap\hat{T}_{a,b}^{j,m'_s})\smallsetminus
(\hat{T}_{a,b}^{j,m'_{s+1}}\cup\cdots\cup\hat{T}_{a,b}^{j,m'_r}).$$

Let us rename the finite set of pairs $(j,m_{\bullet})$ as $\{i\}_{i\in I}$; according to that, let us set
$\hat{T}_{a,b;i}:=\hat{T}_{a,b}^{j,m_{\bullet}}$ and analogously for all the other objects defined so far. In
particular, we have a finite stratification by locally closed subschemes $\hat{T}_{a,b}=\coprod_i\hat{T}_{a,b;i}$. For
any $i=(j,m_{\bullet})$ let us consider the embedding $\hat{\nu}_i:\hat{T}_{a,b;i}\hookrightarrow\hat{T}_{a,b}^j$ and
let us consider the cartesian diagram

\[\begin{tikzpicture}[xscale=1.5,yscale=-1.2]
    \node (A0_0) at (0, 0) {$\hat{P}_{a,b;i}$};
    \node (A0_2) at (2, 0) {$\hat{P}_{a,b}^j$};
    \node (A1_1) at (1, 1) {$\square$};
    \node (A2_0) at (0, 2) {$\hat{T}_{a,b;i}$};
    \node (A2_2) at (2, 2) {$\hat{T}_{a,b}^j.$};
    \path (A0_0) edge [->]node [auto] {$\scriptstyle{\hat{\nu}'_i}$} (A0_2);
    \path (A0_0) edge [->]node [auto] {$\scriptstyle{\hat{\varphi}_{a,b;i}}$} (A2_0);
    \path (A0_2) edge [->]node [auto] {$\scriptstyle{\hat{\varphi}_{a,b}^j}$} (A2_2);
    \path (A2_0) edge [->]node [auto] {$\scriptstyle{\hat{\nu}_i}$} (A2_2);
\end{tikzpicture}\]

Since $\mathcal{H}_{a,b}^j$ commutes with base change, then we have that

$$\hat{P}_{a,b;i}=\mathbb{P}\Big((\hat{\mathcal{H}}_{a,b;i})^{\vee}\Big),$$
where

$$\hat{\mathcal{H}}_{a,b;i}=\mathcal{E}xt^1_{\pi_{\hat{T}_{a,b;i}}}\Big((\hat{p}'_2,\hat{p}_2)^*(\hat{\mathcal{Q}}_2,
\hat{\mathcal{W}}_2)|_{\hat{T}_{a,b;i}},(\hat{p}'_1,\hat{p}_1)^*(\hat{\mathcal{Q}}_1,
\hat{\mathcal{W}}_1)|_{\hat{T}_{a,b;i}}\Big).$$

Now by restricting the sequence (\ref{81}) to $\hat{P}_{a,b;i}$ we get a family of non-split extensions parametrized by
$\hat{P}_{a,b;i}$:

$$0\rightarrow(\hat{\varphi}'_{a,b;i},\hat{\varphi}_{a,b;i})^*(\hat{p}'_1,\hat{p}_1)^*(\hat{\mathcal{Q}}_1,
\hat{\mathcal{W}}_1)\otimes_{\hat{P}_{a,b;i}}\mathcal{O}_{\hat{P}_{a,b;i}}(1)\rightarrow$$
$$\rightarrow(\hat{\mathcal{E}}_{a,b;i},\hat{\mathcal{V}}_{a,b;i})\rightarrow
(\hat{\varphi}'_{a,b;i},\hat{\varphi}_{a,b;i})^*(\hat{p}'_2,\hat{p}_2)^*(\hat{\mathcal{E}}_2,\hat{\mathcal{V}}_2)
\rightarrow 0$$
that is universal on the category of reduced $\hat{T}_{a,b;i}$-schemes.\\

Since $PGL(N_1)\times PGL(N_2)$ acts freely on both $\hat{P}_{a,b;i}$ and $\hat{T}_{a,b;i}$, we have an induced
projective fibration $\varphi_{a,b;i}:G(\alpha_c;n,d,k;n_1,k_1)\rightarrow T_{a,b;i}$. The rest of the proof is
straightforward.\\

If $GCD(n_l,d_l,k_l)=1$ for $l=1,2$, then we have universal families $(\mathcal{Q}_l,\mathcal{W}_l)$ parametrized by
$G_l$ for $l=1,2$, so the construction of the schemes $T_{a,b;i}$ and $P_{a,b;i}$ can be done directly at the level
of $G_1\times G_2$ instead of doing it on $\hat{G}_1\times\hat{G}_2$. Therefore we can also construct universal
families of extensions as in (\ref{71}).
\end{proof}

\begin{corollary}
If $\frac{k_1}{n_1}<\frac{k}{n}$, respectively $\frac{k_1}{n_1}>\frac{k}{n}$, then the scheme $G(\alpha_c;n,d,k;
n_1,d_1,$ $k_1;a,b;j)$ is a subscheme of $G(\alpha_c^+;n,d,k)$, respectively of $G(\alpha_c^-;n,d,k)$. So proposition
\ref{58} gives a scheme theoretic description of the sets $G^{+,2}(\alpha_c;n,d,k)$ and $G^{-,2}(\alpha_c;n,d,k)$ that
were described set theoretically in corollary \ref{15}.
\end{corollary}

\begin{proof}
Let us consider the case when $\frac{k_1}{n_1}<\frac{k}{n}$, the other case is analogous. Let us fix any set of
indices $(a,b;i)$ and the associated short exact sequence (\ref{61}). For every point $\hat{t}:=(\hat{t}_1,\hat{t}_2)
\in\hat{T}_{a,b;i}$ with image $((E_1,V_1),(E_2,V_2))$ in $T_{a,b;i}$ such a sequence restricts to a non-split
extension

$$0\longrightarrow(E_1,V_1)\longrightarrow(E,V)\longrightarrow(E_2,V_2)\longrightarrow 0$$

Now let us consider the conditions on $(n_l,d_l,k_l)$ for $l=1,2$ given in definition \ref{14}. Together with the fact
that the previous sequence is non-split and that $\frac{k_1}{n_1}<\frac{k}{n}$, such conditions imply that $(E,V)$ is
$\alpha_c^+$-stable and $\alpha_c^-$-unstable. In particular, for every $\hat{t}\in\hat{T}_{a,b;i}$ the family
$(\hat{\mathcal{E}}_{a,b;i},\hat{\mathcal{V}}_{a,b;i})$ restricts to a point of $G(\alpha_c^+;n,d,k)$. By the universal
property of $G(\alpha_c^+;n,d,k)$, this induces a morphism

$$\hat{\zeta}_{a,b;i}:\hat{P}_{a,b;i}\rightarrow G(\alpha_c^+;n,d,k).$$

Such a morphism is invariant under the action of $PGL(N_1)\times PGL(N_2)$ on the source. Therefore, it induces a
morphism

$$\zeta_{a,b;i}:G(\alpha_c;n,d,k;a,b;i)\rightarrow G(\alpha_c^+;n,d,k)$$
(if $GCD(n_l,d_l,k_l)=1$ for $l=1,2$, then we can construct $\zeta_{a,b;i}$ directly by using the sequence (\ref{71})).
$\zeta_{a,b;i}$ is an embedding and it has values in $G^{+,2}(\alpha_c;n,d,k)\subset G(\alpha_c^+;n,d,k)$ by lemma
\ref{91}. By construction for different choices of the invariants $(n_1,k_1;a,b;i)$ we get disjoint images in
$G(\alpha_c^+;n,d,k)$. So we conclude.
\end{proof}


\end{document}